\documentclass[a4paper,leqno,12pt]{amsart}
\usepackage{amsmath,amsthm}
\usepackage{amssymb}
\usepackage{xspace}
\usepackage{bm}
\usepackage{amstext}
\usepackage{amsfonts}
\usepackage{graphicx}
\usepackage[mathscr]{euscript}
\usepackage{amscd}
\usepackage{latexsym}
\usepackage{amssymb}
\usepackage{enumerate}
\usepackage{mathrsfs}
\usepackage[cmtip,all]{xy}
\begin{document}
%
%
\theoremstyle{plain}
\swapnumbers
    \newtheorem{thm}[figure]{Theorem}
    \newtheorem{prop}[figure]{Proposition}
    \newtheorem{lemma}[figure]{Lemma}
    \newtheorem{keylemma}[figure]{Key Lemma}
    \newtheorem{corollary}[figure]{Corollary}
    \newtheorem{fact}[figure]{Fact}
    \newtheorem{subsec}[figure]{}
    \newtheorem*{propa}{Proposition A}
    \newtheorem*{thma}{Theorem A}
    \newtheorem*{corb}{Corollary B}
    \newtheorem*{corc}{Corollary C}
\theoremstyle{definition}
    \newtheorem{defn}[figure]{Definition}
    \newtheorem{examples}[figure]{Examples}
    \newtheorem{notn}[figure]{Notation}
    \newtheorem{summary}[figure]{Summary}
     \newtheorem{question}[figure]{Question}
\theoremstyle{remark}
\newtheorem{remark}[figure]{Remark}
        \newtheorem{construction}[figure]{Construction}
        \newtheorem{remarks}[figure]{Remarks}
        \newtheorem{example}[figure]{Example}
        \newtheorem{warning}[figure]{Warning}
    \newtheorem{assume}[figure]{Assumption}
    \newtheorem{ack}[figure]{Acknowledgements}
\renewcommand{\thefigure}{\arabic{section}.\arabic{figure}}
%
%
\newenvironment{myeq}[1][]
{\stepcounter{figure}\begin{equation}\tag{\thefigure}{#1}}
{\end{equation}}
\newcommand{\myeqn}[2][]
{\stepcounter{figure}\begin{equation}
     \tag{\thefigure}{#1}\vcenter{#2}\end{equation}}
\newcommand{\mydiag}[2][]{\myeq[#1]\xymatrix{#2}}
\newcommand{\mydiagram}[2][]
{\stepcounter{figure}\begin{equation}
     \tag{\thefigure}{#1}\vcenter{\xymatrix{#2}}\end{equation}}
\newcommand{\mysdiag}[2][]
{\stepcounter{figure}\begin{equation}
     \tag{\thefigure}{#1}\vcenter{\xymatrix@R=15pt@C=35pt{#2}}\end{equation}}
\newcommand{\mytdiag}[2][]
{\stepcounter{figure}\begin{equation}
     \tag{\thefigure}{#1}\vcenter{\xymatrix@R=25pt@C=50pt{#2}}\end{equation}}
\newcommand{\myudiag}[2][]
{\stepcounter{figure}\begin{equation}
     \tag{\thefigure}{#1}\vcenter{\xymatrix@R=25pt@C=15pt{#2}}\end{equation}}
\newcommand{\myvdiag}[2][]
{\stepcounter{figure}\begin{equation}
     \tag{\thefigure}{#1}\vcenter{\xymatrix@R=19pt@C=30pt{#2}}\end{equation}}
\newcommand{\mywdiag}[2][]
{\stepcounter{figure}\begin{equation}
     \tag{\thefigure}{#1}\vcenter{\xymatrix@R=30pt@C=10pt{#2}}\end{equation}}
\newcommand{\myxdiag}[2][]
{\stepcounter{figure}\begin{equation}
     \tag{\thefigure}{#1}\vcenter{\xymatrix@R=25pt@C=1pt{#2}}\end{equation}}
%
%
\newenvironment{mysubsection}[2][]
{\begin{subsec}\begin{upshape}\begin{bfseries}{#2.}
\end{bfseries}{#1}}
{\end{upshape}\end{subsec}}
\newenvironment{mysubsect}[2][]
{\begin{subsec}\begin{upshape}\begin{bfseries}{#2\vsn.}
\end{bfseries}{#1}}
{\end{upshape}\end{subsec}}
\newcommand{\sect}{\setcounter{figure}{0}\section}
%
%
\newcommand{\wh}{\ -- \ }
\newcommand{\wwh}{-- \ }
\newcommand{\w}[2][ ]{\ \ensuremath{#2}{#1}\ }
\newcommand{\ww}[1]{\ \ensuremath{#1}}
\newcommand{\www}[2][ ]{\ensuremath{#2}{#1}\ }
\newcommand{\wwb}[1]{\ \ensuremath{(#1)}-}
\newcommand{\wb}[2][ ]{\ (\ensuremath{#2}){#1}\ }
\newcommand{\wref}[2][ ]{\ (\ref{#2}){#1}\ }
\newcommand{\wwref}[3][ ]{\ (\ref{#2})-(\ref{#3}){#1}\ }
%
%
\newcommand{\hs}{\hspace*{5 mm}}
\newcommand{\hsm}{\hspace*{2 mm}}
\newcommand{\hsn}{\hspace{2 mm}}
\newcommand{\hsp}{\hspace*{9 mm}}
\newcommand{\vs}{\vspace{5 mm}}
\newcommand{\vsm}{\vspace{3 mm}}
%
%
\newcommand{\hra}{\hookrightarrow}
\newcommand{\xra}[1]{\xrightarrow{#1}}
\newcommand{\xepic}[1]{\xrightarrow{#1}\hspace{-5 mm}\to}
\newcommand{\lora}{\longrightarrow}
\newcommand{\lra}[1]{\langle{#1}\rangle}
\newcommand{\xhra}[1]{\overset{#1}{\hookrightarrow}}
\newcommand{\efp}{\to\hspace{-1.5 mm}\rule{0.1mm}{2.2mm}\hspace{1.2mm}}
\newcommand{\efpic}{\mbox{$\to\hspace{-3.5 mm}\efp$}}
\newcommand{\up}[1]{\sp{(#1)}}
\newcommand{\bup}[1]{\sp{[{#1}]}}
\newcommand{\lo}[1]{\sb{(#1)}}
\newcommand{\lolr}[1]{\sb{\lra{#1}}}
\newcommand{\bp}[1]{\sb{[#1]}}
\newcommand{\rest}[1]{\lvert\sb{#1}}
\newcommand{\sms}[2]{{#1}\wedge{#2}}
\newcommand{\hfsm}[2]{{#1}\ltimes{#2}}
\newcommand{\ii}{\wwb{\infty,1}}
\newcommand{\adj}[2]{\substack{{#1}\\ \rightleftharpoons \\ {#2}}}
%
%
\newcommand{\ab}{\operatorname{ab}}
\newcommand{\coc}{\operatorname{co}}
\newcommand{\cocon}{\operatorname{cocon}}
\newcommand{\Cof}{\operatorname{Cof}}
\newcommand{\Coker}{\operatorname{Coker}}
\newcommand{\colim}{\operatorname{colim}}
\newcommand{\colimit}[1]
{\raisebox{-1.7ex}{$\stackrel{\textstyle\colim}{\scriptstyle{#1}}$}}
\newcommand{\csk}[1]{\operatorname{csk}\sb{#1}}
\newcommand{\cskn}[2]{\csk{#2}\sp{#1}}
\newcommand{\eval}{\operatorname{eval}}
\newcommand{\Fib}{\operatorname{Fib}}
\newcommand{\hocofib}{\operatorname{hocofib}}
\newcommand{\hocolim}{\operatorname{hocolim}}
\newcommand{\ho}{\operatorname{ho}}
\newcommand{\holim}{\operatorname{holim}}
\newcommand{\Hom}{\operatorname{Hom}}
\newcommand{\Id}{\operatorname{Id}}
\newcommand{\id}{\operatorname{Id}}
\newcommand{\Image}{\operatorname{Im}}
\newcommand{\Ker}{\operatorname{Ker}}
\newcommand{\lev}{\operatorname{lev}}
\newcommand{\Obj}[1]{\operatorname{Obj}\,{#1}}
\newcommand{\op}{\sp{\operatorname{op}}}
\newcommand{\Reedy}{\operatorname{Reedy}}
\newcommand{\sk}[1]{\operatorname{sk}\sb{#1}}
\newcommand{\Tot}{\operatorname{Tot}}
%
%
\newcommand{\map}{\operatorname{map}}
\newcommand{\Map}{\operatorname{Map}}
\newcommand{\MAP}{\mathbf{Map}}
\newcommand{\mapa}{\map\sb{\ast}}
%
%
\newcommand{\Ei}[3]{E\sb{#1}\sp{{#2},{#3}}}
\newcommand{\Eis}[2]{E\sb{#1}(#2)}
\newcommand{\Eu}[3]{E\sp{#1}\sb{{#2},{#3}}}
\newcommand{\Eus}[1]{E\sp{#1}}
\newcommand{\Euz}[1]{E\sp{0}({#1})}
\newcommand{\Eot}[2]{\Eu{1}{#1}{#2}}
\newcommand{\Ett}[2]{\Eu{2}{#1}{#2}}
%
%
\newcommand{\Cs}{C\sb{\ast}}
\newcommand{\Cus}{C\sp{\ast}}
\newcommand{\hC}{\hat{\C}}
\newcommand{\Ds}{D\sb{\ast}}
\newcommand{\Dus}{D\sp{\ast}}
\newcommand{\cM}[1]{C\sb{#1}}
\newcommand{\cZ}[1]{Z\sb{#1}}
\newcommand{\df}[1]{\partial\sb{#1}}
\newcommand{\wdf}[1]{\widehat{\partial}\sb{#1}}
\newcommand{\oG}[1]{\overline{G}\sb{#1}}
\newcommand{\oS}[1]{\boxtimes S\sp{#1}}
\newcommand{\odisc}[1]{\boxtimes e\sp{#1}}
\newcommand{\sA}[1]{\Sigma\sp{#1}\bA}
\newcommand{\CsA}[1]{C\Sigma\sp{#1}\bA}
\newcommand{\Ws}{W\sb{\ast}}
%
%
\newcommand{\fB}{\mathfrak{B}}
\newcommand{\fC}{\mathfrak{C}}
\newcommand{\fd}{\mathfrak{d}}
\newcommand{\fy}{\mathfrak{h}}
\newcommand{\hy}{\hat{\fy}}
%
%
\newcommand{\A}{\mathcal{A}}
\newcommand{\B}{\mathcal{B}}
\newcommand{\C}{\mathcal{C}}
\newcommand{\D}{\mathcal{D}}
\newcommand{\E}{\mathcal{E}}
\newcommand{\F}{\mathcal{F}}
\newcommand{\G}{\mathcal{G}}
\newcommand{\HH}{\mathcal{H}}
\newcommand{\J}{\mathcal{J}}
\newcommand{\LLc}{\mathcal{L}}
\newcommand{\LH}{\LLc\sb{H}}
\newcommand{\LHX}{\LH(X)}
\newcommand{\LK}{\LLc\sb{K}}
\newcommand{\LW}{\LLc\sb{\W}}
\newcommand{\LWX}{\LW(X)}
\newcommand{\LR}{\LLc\sb{\Reedy}}
\newcommand{\M}{\mathcal{M}}
\newcommand{\PPc}{\mathcal{P}}
\newcommand{\R}{\mathcal{R}}
\newcommand{\Ss}{\mathcal{S}}
\newcommand{\Sa}{\Ss\sb{\ast}}
\newcommand{\Sr}{\Ss\sp{\red}}
\newcommand{\U}{\mathcal{U}}
\newcommand{\W}{\mathcal{W}}
\newcommand{\Z}{\mathcal{Z}}
%
%
\newcommand{\Ab}{\mbox{\sf Ab}}
\newcommand{\Ch}{\mbox{\sf Ch}}
\newcommand{\Chn}[2]{\Ch\sp{#1}\sb{#2}}
\newcommand{\enk}[2]{\eta\sp{#1}\sb{#2}}
\newcommand{\lnk}[2]{\lambda\sp{#1}\sb{#2}}
\newcommand{\rnk}[2]{\rho\sp{#1}\sb{#2}}
\newcommand{\tnk}[2]{\tau\sp{#1}\sb{#2}}
\newcommand{\cCh}{\coc\!\Ch}
\newcommand{\DK}{\mbox{\sf DK}}
\newcommand{\Kan}{\mbox{\sf Kan}}
\newcommand{\LKE}{\mbox{\sf LKE}}
\newcommand{\Set}{\mbox{\sf Set}}
\newcommand{\sSet}{\mbox{\sf sSet}}
\newcommand{\Seta}{\Set\sb{\ast}}
\newcommand{\sSeta}{\sSet\sb{\ast}}
\newcommand{\Cat}{\mbox{\sf Cat}}
\newcommand{\sCat}{\mbox{\sf sCat}}
\newcommand{\sCata}{\sCat\sb{\ast}}
\newcommand{\Top}{\mbox{\sf Top}}
\newcommand{\Topa}{\Top\sb{\ast}}
\newcommand{\Tw}{\mbox{\sf Tw}}
%
%
\newcommand{\mA}{\mathscr{A}}
\newcommand{\mB}{\mathscr{B}}
\newcommand{\mC}{\mathscr{C}}
\newcommand{\mM}{\mathscr{M}}
\newcommand{\mW}{\mathscr{W}}
\newcommand{\mX}{\mathscr{X}}
\newcommand{\mY}{\mathscr{Y}}
%
%
\newcommand{\CC}{\mathbb C}
\newcommand{\FF}{\mathbb F}
\newcommand{\II}{\mathbb I}
\newcommand{\NN}{\mathbb N}
\newcommand{\QQ}{\mathbb Q}
\newcommand{\RR}{\mathbb R}
\newcommand{\ZZ}{\mathbb Z}
%
%
\newcommand{\Del}{\mathbf{\Delta}}
\newcommand{\Dell}[1]{\Delta\sb{\leq{#1}}}
\newcommand{\Du}{\Del\sp{\bullet}}
\newcommand{\rDel}{\sb{+}\!\Delta}
\newcommand{\rDell}[1]{\sb{+}\Delta\sb{#1}}
\newcommand{\rDele}[1]{\sb{+}\Delta\sb{\leq{#1}}}
\newcommand{\Deln}[1]{\Delta\sp{#1}}
\newcommand{\Dop}{\Delta\op}
\newcommand{\rDop}{\rDel\op}
\newcommand{\rDopl}[1]{\rDop\sb{\leq{#1}}}
\newcommand{\res}{\operatorname{res}}
\newcommand{\cpX}{c\sp{+}X}
\newcommand{\cpn}[2]{c\sp{+}\sb{#1}{#2}}
\newcommand{\cpnX}[1]{\cpn{#1}{X}}
\newcommand{\spX}{s\sb{+}X}
\newcommand{\spS}{s\sb{+}\Sa}
\newcommand{\spn}[2]{s\sb{+}\sp{#1}{#2}}
\newcommand{\spnX}[1]{\spn{#1}{X}}
%
%
\newcommand{\bA}{{\mathbf A}}
\newcommand{\tA}{\widetilde{\bA}}
\newcommand{\bB}{{\mathbf B}}
\newcommand{\bC}{{\mathbf C}}
\newcommand{\bD}{{\mathbf D}}
\newcommand{\Dnmp}[3]{\bD\sp{#1}\sb{#2}(#3)}
\newcommand{\bDs}{\bD\sb{\ast}}
\newcommand{\tDs}{\widetilde{\bD}\sb{\ast}}
\newcommand{\bM}{{\mathbf M}}
\newcommand{\bQ}{{\mathbf Q}}
\newcommand{\bS}[1]{{\mathbf S}\sp{#1}}
\newcommand{\bU}{{\mathbf U}}
\newcommand{\bV}{{\mathbf V}}
\newcommand{\bW}{{\mathbf W}}
\newcommand{\oW}[1]{\overline{\bW}\sp{#1}}
\newcommand{\bX}{{\mathbf X}}
\newcommand{\hX}{\widehat{\bX}}
\newcommand{\tX}{\widetilde{\bX}}
\newcommand{\oX}[1]{\overline{\bX}\sb{#1}}
\newcommand{\bY}{{\mathbf Y}}
\newcommand{\hY}{\widehat{\bY}}
\newcommand{\bZ}{{\mathbf Z}}
%
%
\newcommand{\vare}{\varepsilon}
\newcommand{\var}[1]{\vare\sb{#1}}
\newcommand{\gam}[1]{\gamma\lo{#1}}
%
%
\newcommand{\bd}{\mathbf{d}\sb{0}}
\newcommand{\cod}[1]{c\sb{+}({#1})\sb{\bullet}}
\newcommand{\cu}[1]{c({#1})\sp{\bullet}}
\newcommand{\Lu}{\Lambda\sp{\bullet}}
\newcommand{\szero}[1]{\sigma\sp{\ast}(#1)}
\newcommand{\Cd}[1]{\bC[{#1}]\sb{\bullet}}
\newcommand{\Dup}[1]{\bD\sp{\bullet}\bp{#1}}
\newcommand{\Fu}{F\sp{\bullet}}
\newcommand{\Mu}{M\sp{\bullet}}
\newcommand{\Qu}{\bQ\sp{\bullet}}
\newcommand{\Ud}{\bU\sb{\bullet}}
\newcommand{\Vd}{V\sb{\bullet}}
\newcommand{\Wd}{\bW\sb{\bullet}}
\newcommand{\Wu}{\bW\sp{\bullet}}
\newcommand{\wu}{w\sp{\bullet}}
\newcommand{\wwd}{w\sb{\bullet}}
\newcommand{\Xd}{\bX\sb{\bullet}}
\newcommand{\xd}{x\sb{\bullet}}
\newcommand{\xu}{x\sp{\bullet}}
\newcommand{\hXd}{\hX\sb{\bullet}}
\newcommand{\Yd}{\bY\sb{\bullet}}
\newcommand{\yd}{y\sb{\bullet}}
\newcommand{\yu}{y\sp{\bullet}}
\newcommand{\Zd}{\bZ\sb{\bullet}}
\newcommand{\zd}{z\sb{\bullet}}
\newcommand{\zu}{z\sp{\bullet}}
\newcommand{\bk}{[\mathbf{k}]}
\newcommand{\bn}{[\mathbf{n}]}
\newcommand{\bmm}{[\mathbf{m}]}
\newcommand{\bnm}{[\mathbf{n}-\mathbf{m}]}
\newcommand{\bnmp}{[\mathbf{n}-\mathbf{m}+\mathbf{1}]}
\newcommand{\bone}{[\mathbf{1}]}
\newcommand{\lone}{\lra{\mathbf{1}}}
\newcommand{\brp}{[\mathbf{r}+\mathbf{1}]}
\newcommand{\Po}[1]{P\sp{#1}}
%
%
\title{Localization of $(\infty,1)$-categories and spectral sequences}
%
%
\author[D.~Blanc]{David Blanc}
\address{Department of Mathematics\\ University of Haifa\\ 3498838 Haifa\\ Israel}
\email{blanc@math.haifa.ac.il}
\author[N.J.~Meadows]{Nicholas J.\ Meadows}
\address{Department of Mathematics\\ University of Haifa\\ 3498838 Haifa\\ Israel}
\email{njmead81118@gmail.com}
\date{\today}

\begin{abstract}
We describe two types of localization for \wwb{\infty,1}categories which determine the
successive terms in the homotopy spectral sequence of a (co)simplicial object.
\end{abstract}

\maketitle

\setcounter{section}{0}

\section*{Introduction}
\label{cint}

The goal of this paper is to try to understand the information contained in the
successive terms of spectral sequences, from the point of view of
\wwb{\infty,1}categories. Although many of the spectral sequences in common use are
stable, and may be described completely in terms of towers of spectra or filtered
chain complexes, we concentrate here on the unstable version, and more specifically
the homotopy spectral sequence of a (co)simplicial object.

The main reason is that in this paper we are only concerned with the differentials
(and thus the  successive finite terms) in spectral sequences \wh and these
are more complicated, and thus more illuminating, in the unstable version.
Moreover, in many cases of interest, the differentials in a stable spectral sequence
are determined by those in an associated unstable spectral sequence (e.g., for
the Adams spectral sequence).

In \cite{CELWhitM},  Cirici, Egas Santander, Livernet, and Whitehouse analyze the
spectral sequence of a filtered complex of $R$-modules, for any commutative
ring $R$, in a similar spirit. However, they do this in the context of model categories,
while the setting of \wwb{\infty,1}categories seems more appropriate for our purposes.

Our goal here is to understand two seemingly contradictory phenomena:
on the one hand, in the successive terms of a spectral sequence we discard extraneous
information, as we see from the fact that a map \w{f:\xd\to\yd} of simplicial objects
inducing an isomorphism in the \ww{E\sp{r}}-term necessarily induces an isomorphism in
the \ww{E\sp{r+1}}-term, but not conversely.  On the other hand, we need
more (and higher order) information to compute \w{d\sp{r+1}} from the given \w{\xd}
than we do for \w[.]{d\sp{r}} The reason is that as we proceed in the spectral sequence,
we require less data from the original (co)simplicial space, but gain knowledge of the
abutment.

This suggests that we need two types of localizations in order to analyze
the successive terms of the spectral sequence of a (co)simplicial object
in an \wwb{\infty,1}category, which may be described as follows\vsm:

The first part of the paper is a study of Cisinski's quasi-category version
of a (co)fibration category, reviewed in Section \ref{cqccf},
and the left Bousfield localization with respect to a set of maps 
(in Section \ref{clbl}) or a set of objects (in Section \ref{clocth}).
In Section \ref{crbl} we construct right Bousfield localizations with respect
a set of objects in a quasi-category equipped with classes of weak equivalences and
either cofibrations or fibrations.

The second part of the paper applies this theory to spectral sequences:
in Section \ref{csssso} we recall from \cite{BMeadS} the description of the homotopy
spectral sequence of a simplicial object \w{\xd} in an \wwb{\infty,1}category $X$,
with respect to a given homotopy cogroup object $\hy$ in $X$, and show how to
reinterpret the construction in terms of chain complexes in spaces,
yielding a convenient diagrammatic description of the differentials
(see \S \ref{sdiagdesc}).

In Section \ref{cssl} we use this description to show that the
\ww{d\sp{r}}-differential in the spectral sequence depends only on the Postnikov
section \w[.]{\Po{r-1}\Omega\sp{p}\Map\sb{X}(\hy,\xd)} This dependence can be made
more precise using a certain countable collection \w{\HH\sp{r}} of finite segments
of simplicial objects \w{H(n,m,\Sigma\sp{p}\hy)} of length \w{m\leq r-1} (see
\S \ref{dzmnp}). We define the $r$-\emph{stem} for \w{\lra{\xd,\hy}} to be the system
\begin{myeq}\label{eqrstem}
\{\Po{m}\Map(H(n,m,\Sigma\sp{p}\hy),\tau\sp{\ast}\xd)\}
\sb{H(n,m,\Sigma\sp{p}\hy)\in\HH\sp{r}}~,
\end{myeq}
\noindent and show:
\begin{thma}\label{nthma}
For each \w[,]{r\geq 2} the \ww{E\sp{r}}-term of the spectral sequence
associated to \w{\lra{\xd,\hy}} is determined by its \wwb{r-2}stem.
\end{thma}
\noindent See Theorem \ref{tstem} below\vsm.

We then define a pair of left and right Bousfield localizations on the product
\w{Z\sp{r}} of chain complex segments in \w{\Sa} corresponding to \w[,]{\HH\sp{r}}
yielding the $r$-th Postnikov localization \w[,]{\PPc\sp{r}:Z\sp{r}\to Z\sp{r}}
and deduce:
\begin{corb}\label{ncorb}
A \ww{\PPc\sp{r}}-equivalence in \w{Z\sp{r}} induces an
isomorphism of the \ww{E\sp{r+2}}-terms of the  associated
spectral sequences.
\end{corb}
\noindent See Corollary \ref{cstempr}\vsm.

In fact, the spectral sequence of \w{\xd} depends only on the underlying
restricted simplicial object in \w{\spX} (forgetting the degeneracies).
We can define the right Bousfield localization of \w{\spX} with respect to
\w[,]{\HH\sp{r}} and show:
\begin{corc}\label{ncorc}
The \ww{\HH\sp{r}}-equivalences in \w{\spX} induce \ww{E\sp{r}}-isomorphisms of
the associated spectral sequences.
\end{corc}
\noindent See Corollary \ref{cerloc}\vsm.

Section \ref{cssrfcs} is devoted to a detailed analysis of the
spectral sequence of a cosimplicial object \w{\xu} in a quasi-category $X$
(which was only sketched in \cite[\S 9]{BMeadS}), again providing a diagrammatic
description of the differentials (see \S \ref{sdiagdes}). This is used in
Section \ref{ccsloc} to describe the cosimplicial version of $n$-stems, the Postnikov
localization \w[,]{\PPc\sp{r}} and the right Bousfield localization \w[,]{R\sb{H}}
satisfying analogues of Theorem \textbf{A} and Corollaries \textbf{B} and \textbf{C}
above (namely, Theorem \ref{tcstem} and Corollaries  \ref{cctempr}
and \ref{thm8.6} below).

\begin{mysubsection}{Notation and conventions}
\label{snac}
The category of sets is denoted by \w[,]{\Set} and that of pointed sets by \w[.]{\Seta}
Similarly, \w{\Top} denotes the category of topological spaces, and \w{\Topa} that
of pointed spaces.

Let $\Delta$ denote the category of non-empty finite ordered sets and order-preserving
maps, so that a functor \w{F:\Dop\to\C} is a \emph{strict simplicial object}
in the category $\C$, and the category of such is denoted by \w[.]{\C\sp{\Dop}}
Similarly, a functor \w{G:\Delta\to\C} is a \emph{strict cosimplicial object},
and the category of such is denoted by \w[.]{\C\sp{\Delta}}
However, the category \w{\Set\sp{\Dop}} of simplicial sets, called \emph{spaces},
is denoted simply by $\Ss$, and that of pointed simplicial sets by
\w[.]{\Sa:=\Seta\sp{\Dop}}  We denote the category of small categories by \w[,]{\Cat}
with \w{B:\Cat\to\Ss} the ordinary \emph{nerve} functor.

The object \w{0<1\dotsc<n} in $\Delta$ is denoted by \w[,]{\bn}
and for functors \w{G:\Dop\to\C} or \w{H:\Delta\to\C} we write
\w{G\sb{n}} for \w{G(\bn)} and \w{H\sp{n}} for \w[.]{H(\bn)}
However, we shall use the notation \w{\lone} to refer to the
single arrow category \w{0\to 1} when it is necessary to
distinguish it from the corresponding object of $\Delta$.

We let \w{\rDel} denote the subcategory of $\Delta$ with the same objects but only
monic maps (thus representing \emph{restricted} (co)simplicial objects).
For \w[,]{m < n} we denote by \w{\Delta\sb{m, n}} (respectively, \w[)]{\rDell{m,n}}
the full subcategory of \w{\Delta} (respectively, \w[)]{\rDel} consisting of the objects
\w{\bk} with \w[.]{m \leq k\leq n} We abbreviate \w{\Delta\sb{0,n}} to \w{\Dell{n}}
and \w{\rDell{0,n}} to \w[.]{\rDele{n}} 

If $\C$ has the necessary (co)limits, the inclusion \w{i\sb{n}:\Dell{n}\hra\Delta}
induces \w{i\sb{n}\sp{\ast}:\C\sp{\Dop}\to\C\sp{\Dell{n}\op}} with
left adjoint \w[,]{i\sb{n}'} and the $n$-\emph{skeleton} functor is
\w[.]{\sk{n}=i\sb{n}'\circ i\sb{n}\sp{\ast}:\C\sp{\Dop}\to\C\sp{\Dop}}
Similarly, the $n$-\emph{coskeleton} functor
\w{\csk{n}=i\sb{n}''\circ i\sb{n}\sp{\ast}:\C\sp{\Dop}\to\C\sp{\Dop}} is
defined using the right adjoint \w{i\sb{n}''} of \w[.]{i\sb{n}\sp{\ast}}
Variants of these functors exist for \w[.]{j\sb{n}:\rDele{n}\hra\rDel}
Note that for \w[,]{\C=\Set} \w{\csk{n+1}A} is a model for the $n$-th Postnikov
section of a fibrant simplicial set $A$ (see \cite[VI, \S 3.4]{GJarS}).

The standard $n$-dimensional simplex in $\Ss$, represented by \w[,]{\bn} is denoted by
\w[,]{\Deln{n}} its boundary (obtained by omitting the non-degenerate simplex in
\w[)]{\sigma\sb{n}\in\Deln{n}\sb{n}} by \w[,]{\partial\Deln{n}} and the
\wwb{n,k}horn (obtained by further omitting \w[)]{d\sb{k}\sigma\sb{n}}
by \w{\Lambda\sb{k}\sp{n}} (see \cite[I, \S 1]{GJarS}).

A \emph{quasi-category} is a simplicial set $X$ in which, for each \w[,]{0<k< n}
all liftings of the form
\mydiagram[\label{eqqcat}]{
\Lambda\sb{k}\sp{n} \ar[r] \ar[d] & X \\
\Deln{n} \ar@{.>}[ur] &
 }
\noindent exist (see \cite{JoyQ,JoyQA}).

If $X$ is a quasi-category, we write \w{sX:= X\sp{B(\Dop)}} for category of
\emph{simplicial objects} in $X$, \w{\spX := X\sp{B(\rDop)}} for the category
of \emph{restricted simplicial objects} in $X$, and
\w[.]{\spnX{m, n}:=X\sp{B(\rDop\sb{m, n})}} The truncation
functors \w{sX\to sX\sb{n}} and \w{\spX\to\spnX{n}}
(corresponding to \w{\Dell{n}\hra\Delta} and \w[)]{\rDele{n}\hra\rDel}
will be denoted by \w[.]{\tnk{\ast}{n}}

Dually, we write \w{cX:=X\sp{B(\Del)}} for the category \emph{cosimplicial objects}
in $X$, \w{\cpX:= X\sp{B(\rDel)}} for that of
\emph{restricted cosimplicial} objects, and
\w[.]{\cpnX{m,n}:= X\sp{B(\rDell{m,n})}}

The category of simplicial categories (that is, those enriched in $\Ss$, which we
will usually indicate by $\mX$, $\mY$, and so on)
will be denoted by \w[,]{\sCat} and that of pointed simplicial categories
(enriched in \w[)]{\Sa} by \w[.]{\sCata}
In particular, we write \w{\mapa(x,y)} for the
standard function complex in \w{\Sa} or \w{\Topa} (see \cite[I, \S 1.5]{GJarS}).

When we have a simplicial model category with its associated simplicial enrichment, we
denote the former by $\bC$, say, and the latter by \w[.]{\mC} As for a quasi-category $X$,
we write \w[,]{\ho X} \w[,]{\ho\bC} or \w{\ho\mC} for the associated homotopy category.
\end{mysubsection}

\begin{remark}\label{rinfc}
The category of simplicial sets admits a model category structure
in which the fibrant objects are quasi-categories and the weak equivalences are
\emph{Joyal equivalences} (see \cite[\S 2.2]{LurieHTT} and \cite{JoyQ,JoyQA}).
Similarly, there is a model category structure on \w{\sCat} in which the fibrant
objects are categories enriched in Kan complexes, and the weak equivalences are
Dwyer-Kan equivalences (see \cite{BergM}).

All the definitions and results in this paper could be stated in any of the known
models of \wwb{\infty,1}categories (see, e.g., \cite{BergH}), and could in fact be
presented in a model-independent way, using To{\"{e}}n's axiomatic formulation
(see \cite[\S 4]{ToenA}), for example, as was done in \cite{BMeadS}.
However, in the interests of concreteness we restrict attention here to the above two
models, using when needed the Quillen equivalence
\begin{myeq}\label{eqquileq}
\fC\colon\sSet\leftrightarrows\sCat\colon\fB
\end{myeq}
\noindent between the Bergner and Joyal model categories
(see \cite[Theorem 2.2.5.1]{LurieHTT}).

The right adjoint $\fB$ is the \emph{homotopy coherent nerve}, while
we can think of \w{\fC(X)} as a strict model for the \wwb{\infty,1}category $X$,
described quite explicitly in \cite{DSpivR}.
\end{remark}

%
%
\sect{Quasi-Categories with a Class of Fibrations}
\label{cqccf}

In this section, we review Cisinski's notion of a quasi-category equipped with a
class of fibrations and weak equivalences, serving as the $\infty$-category version
of Brown's fibration categories (see \cite{KBrowA}, and compare \cite{BaueA}).
This material is largely taken from \cite{CisiH}.

\begin{defn}[\protect{\cite[Definition 7.4.6]{CisiH}}]
\label{def1.1}
Let $X$ be a quasi-category with a fixed terminal object $e$. A subcollection
\w{\Fib \subseteq X\sb{1}} is called a \emph{class of fibrations} if it satisfies the
following properties:
\begin{enumerate}
\renewcommand{\labelenumi}{(\arabic{enumi})~}
\item It contains all the identity maps and is closed under composition;
\item Pullbacks of fibrant objects (that is, objects such that the canonical map
  \w{x \to e} is in \w[)]{\Fib} exist;
\item The pullback of a fibration between fibrant objects by a map with fibrant source
  is a fibration.
\end{enumerate}
\end{defn}

\begin{defn}[\protect{\cite[Definition 7.4.12]{CisiH}}]
 \label{def1.2}
 A \emph{quasi-category with fibrations and weak equivalences} is a triple
 \w{(X, \W, \Fib)} consisting of a quasi-category $X$, a class of
  fibrations \w{\Fib \subseteq X\sb{1}} as above, and a subcategory of weak equivalences
\w{\W \subseteq X} satisfying the $2$-out-of-$3$ property, such that:
\begin{enumerate}
\renewcommand{\labelenumi}{(\arabic{enumi})~}
\item Given a pullback diagram
\mydiagram[\label{eqpbd}]{
x \ar[r]\sp{f'} \ar[d] & y \ar[d]\sp{g} \\
z \ar[r]\sb{f} & w
}
such that  the objects $y$, $z$, and $w$ are fibrant with $f$ is a weak equivalence and
a fibration, then the map \w{f'} is also a weak equivalence and a fibration.
\item Every map \w{f : x \to y} with fibrant codomain can be factored as a
  weak equivalence followed by a fibration.
\end{enumerate}
\end{defn}

\begin{example}\label{exam1.3}
If $X$ is a quasi-category with finite limits, then \w{(X, X,\J(X))} has the structure
of a quasi-category with weak equivalences and fibrations, where \w{\J(X)}
is the maximal Kan subcomplex of $X$.
\end{example}

\begin{example}\label{exam1.4}
Let $C$ be a category. If \w{(C,\Fib,\W)} is a category of fibrant objects in the sense
of \cite{KBrowA}, then \w{(B(C), B(\Fib), B(\W))} is a quasi-category with fibrations
and weak equivalences.
\end{example}

A category $I$ is called \emph{cycle-free} if there are no non-identity endomorphisms
in $I$. An object $i$ of a cycle-free category $I$ is said to be of \emph{finite length}
if there is an integer $n$ such that for each string of non-identity morphisms
\w[,]{i\sb{0} \xrightarrow{f\sb{0}} \cdots \xrightarrow{f\sb{m-1}} i\sb{m} = i}
necessarily \w[.]{m \le n} The smallest such $n$ is called the \emph{length} of $i$,
denoted by \w[.]{\ell(i)}
A \emph{directed category} is a cycle-free category $I$ in which each object has
finite length. Such a category is \emph{filtered} by the full subcategories \w{I\sp{(n)}}
consisting of objects of length at most $n$, and we set
\w[.]{\partial I\sp{(n)}:= I\sp{(n)}-\{ x:\ell(x) =n\}}
Given an object \w[,]{\xd\in X\sp{B(I\op)}} we write
$$
\partial x\sb{i} :=
\underset{j \in \partial(I/i)\sp{\ell(i)}}{\underset{ \longleftarrow }\lim} x\sb{j}~.
$$

Assume given a quasi-category $X$ equipped with a class of fibrations \w{\Fib} and
a directed category $I$. A map \w{f : \xd \to \yd}
in \w{X\sp{B(I\op)}} is a \emph{Reedy fibration} if for each \w[,]{i \in I} the map
\w{x\sb{i} \to y\sb{i} \times\sb{\partial y\sb{i}} \partial x\sb{i}}
is a fibration in $X$.

\begin{thm}[\protect{\cite[Theorem 7.4.20]{CisiH}}]
\label{thm1.5}
Let $X$ be a quasi-category with finite limits equipped with
classes \w{\Fib} of fibrations and $\W$ of weak equivalences, and let $I$ be a
finite directed category. The Reedy fibrations and levelwise weak equivalences
then give \w{X\sp{B(I\op)}} the structure of a category with weak equivalences
and fibrations.
\end{thm}

By using slightly stronger hypotheses, we obtain the following generalization:

\begin{corollary}\label{cor1.6}
  If $X$ as in Theorem \ref{thm1.5} has countable limits, then the Reedy fibrations
  and levelwise weak equivalences give \w{\spX} the
structure of a category with weak equivalences and fibrations.
\end{corollary}

\begin{proof}
The fact that the Reedy fibrations form a class of fibrations follows from
\cite[Proposition 7.4.10]{CisiH}. Similarly, condition (1) of \ref{def1.2} is
\cite[Proposition 7.4.18]{CisiH}. It remains to verify is the existence of factorizations:

For each \w[,]{n \in \NN} there is an adjunction
$$
i\sb{n}\sp{\ast}\colon\spnX{0, n}\leftrightarrows \spX\colon(i\sb{n})\sb{\ast}
$$
where \w{i\sb{n}\sp{\ast}} is the restriction of presheaves and \w{(i\sb{n})\sb{\ast}}
is the right Kan extension (which exists by \cite[Proposition 6.4.9]{CisiH},
since $X$ has countable limits).

Let \w{f:x\to y} be a map with (Reedy) fibrant codomain. Using the construction
of \cite[Proposition 7.4.19]{CisiH}, we can produce compatible factorizations of
\w{i\sb{n}\sp{\ast}(f)} as a levelwise weak equivalence and a Reedy fibration, denoted by
\w[.]{g\sb{n}\circ h\sb{n}} Then
$$
\lim((i\sb{n})\sb{\ast}(g\sb{n})) \circ \lim((i\sb{n})\sb{\ast}(h\sb{n}))
$$
gives the required factorization.
\end{proof}

\begin{remark}\label{rmk1.7}
Theorem \ref{thm1.5} can be further generalized to the case where $I$ is an arbitrary
directed category by an inductive argument on the length of objects, using the argument of
Corollary \ref{cor1.6}, as long as $X$ has enough limits to guarantee existence of
the right Kan extension (using \cite[Proposition 6.4.9]{CisiH}).
\end{remark}

\begin{prop}\label{pinduce}
Suppose given a quasi-category $Y$, \w{(X,\Fib,\W)} as in Theorem \ref{thm1.5}, and 
a functor of quasi-categories \w{F:Y\to X} such that:
\begin{enumerate}
\renewcommand{\labelenumi}{(\alph{enumi})~}
\item $F$ is essentially surjective;
\item \w{\ho F} is full;
\item $F$ preserves pullbacks and \w[,]{F(e')=e} for \w{e\in X} and \w{e'\in Y}   
terminal objects.
\end{enumerate}
Then \w{(Y,F\sp{-1}(\Fib),F\sp{-1}(\W))} has the structure of a quasi-category with
fibrations and weak equivalences.
\end{prop}

\begin{proof}
By the $2$-out-of-$3$ property, \w{F\sp{-1}(\W)} contains all identity maps and is
closed under composition. Since $F$ preserves pullbacks, it takes pullbacks of fibrations
to pullbacks of fibrations, and thus satisfies (2) and (3) of Definition \ref{def1.1}.
$F$ also preserves diagrams of the shape found in Definition \ref{def1.1}(1), and
thus satisfies that requirement, too. Finally, to verify Definition \ref{def1.2}(2),
let \w{f:c\to d} be a map in $Y$ with \w{F(d)} fibrant, and factor \w{F(f)} as a weak
equivalence followed by a fibration:
$$
F(c) \xrightarrow{a} d' \xrightarrow{b} F(d)~.
$$
\noindent By hypotheses (a) and (b), we can find a diagram
\mydiagram[\label{eqtwosquare}]{
F(c) \ar[r]\sb{\Id} \ar[d]_{F(a')} & F(c) \ar[d]_{a} \\
F(d'') \ar[r]\sb{w}  \ar[d]_{F(b')} & d' \ar[d]_{b} \\
F(d) \ar[r]\sb{\Id} & F(d)~,
}
in which the map \w{F(b')} is an equivalence and the bottom square is a pullback
(since both its horizontal morphisms are equivalences). Thus \w{F(b')} is a fibration,
and \w{b' \circ a'} gives the required factorization of $f$.
\end{proof}

%
%
\sect{Left Bousfield Localization}
\label{clbl}

In this section, we review the theory of localizations of $\infty$-categories
from \cite{CisiH}. Given a locally presentable quasi-category $X$, and a set
of maps $K$ in $X$, we construct the so-called $K$-equivalence structure,
in which the weak equivalences are $K$-equivalences.

\begin{defn}\label{dlocal}
If $\W$ is a subcategory of a quasi-category $X$ satisfying
the $2$-out-of-$3$ property, the \emph{localization} of $X$ at $\W$ is an object
\w{\LWX} corepresenting the functor
$$
(-)\sp{X} \times\sb{(-)\sp{W}} \times\J(-)\sp{W}
$$
\noindent (see \S \ref{exam1.3}).
Any functor of the form \w{\LWX} is called a \emph{localization functor}.
The image of \w{\id\sb{\LWX}} under 
$$
\LWX)\sp{\LWX}\simeq(\LWX)\sp{X}\times\sb{(\LWX)\sp{\W}}\times\J(\LWX)\sp{\W}\to
(\LWX)\sp{X}
$$
is the \emph{localization map} \w[.]{X \to \LWX}
This has an evident universal property among all maps of quasi-categories
\w{X\to Y} which takes maps in $\W$ to equivalences.
\end{defn}

\begin{thm}[\protect{\cite[Proposition 7.1.4]{CisiH}}]
\w{\LWX} exists for all choices of $X$ and $\W$.
\end{thm}

\begin{example}\label{exam2.2}
Suppose that \w{\mW\subseteq\mX} is an inclusion of (fibrant) simplicial
categories, with underlying $1$-categories \w[.]{W\subseteq X} Then by
\cite[Proposition 1.2.1]{HiniDK}, we have an equivalence of quasi-categories
$$
\LLc\sb{\fB(\mW)}(\fB(\mX))~\simeq~\fB\LH(X,W)~,
$$
\noindent where \w{\LH(X,W)} is the fibrant replacement in the Bergner structure
of the hammock localization in the sense of Dwyer and Kan (\cite{DKanL}).

In particular, suppose that $C$ is the underlying category of a simplicial model
category $\bC$, with underlying simplicial category $\mC$ and underlying category
of weak equivalences \w[.]{W\subseteq C} Then we have a weak equivalence
$$
\fB(\mC) \simeq \fB(\LH(C,W)) \simeq \LLc\sb{BW}(BC)
$$
\noindent by the preceding paragraph and \cite[Proposition 4.8]{DKanF}. In particular,
we can interpret this as saying that \w{\LLc\sb{BW}(BC)} \emph{presents}
the model category $\bC$.
\end{example}

\begin{defn}\label{def2.3}
A \emph{left Bousfield localization} of a quasi-category $X$ is a localization
functor \w{X \to Y} with a fully faithful right adjoint. We call a left Bousfield
localization \emph{left exact} if it preserves finite limits.
Dually, a \emph{right Bousfield localization} of a quasi-category $X$ is a localization
functor \w{X \to Y} with a fully faithful left adjoint.
\end{defn}

\begin{remark}\label{rlrbl}
By \cite[Proposition 5.2.7.6]{LurieHTT}, left Bousfield localizations of simplicial
model categories give rise to left Bousfield localizations of quasi-categories.
\end{remark}

\begin{defn}\label{con2.5}
Suppose that $K$ is a small class of maps. Then the \emph{left Bousfield localization}
of $X$ at $K$, \w[,]{X\to\LK\sp{\cocon}(X)} is the map universal among
cocontinuous maps that take elements of $K$ to equivalences.

We will call any map whose image under left Bousfield localization
\w{X\to\LK\sp{\cocon}(X)} is an equivalence a
\emph{$K$-equivalence}.
\end{defn}

Throughout the rest of the section, we will fix a quasi-category with weak equivalences
and fibrations \w[,]{(X,\Fib,\W)} a small collection of maps $K$,
and assume the following:

\begin{assume}\label{ass3.1}
\begin{enumerate}
\renewcommand{\labelenumi}{(\arabic{enumi})~}
\item The images of the domain and codomain of the maps $K$
under the localization map are compact and connected.
\item $X$ and \w{\LWX} are locally presentable, and the localization map
\w{X \to \LWX} is accessible.
\item $\W$ is \emph{saturated}: that is,  a morphism of $X$ is in $\W$
  if and only if its image under the localization map is invertible.
\end{enumerate}
\end{assume}

\begin{example}\label{examxxx}
Suppose that $\bC$ is an \emph{excellent} simplicial model category (in the sense of
\cite[Definition A.3.2.16]{LurieHTT}) with underlying category $C$, subcategory of
weak equivalences $\W$, and underlying simplicial category $\mC$. Many known
examples of model categories (the Kan model structure on simplicial sets,
the Jardine model structure on simplicial presheaves, and so on) are
excellent. We claim that the structure \w{(B(C\sp{f}), \Fib, B(\W\sp{f}))} given by
Example \ref{exam1.4} automatically satisfies Assumptions \ref{ass3.1}(2)-(3).
Note that Assumption \ref{ass3.1}(1) involves choosing a collection of maps between
`homotopy compact and connected objects' such that simplicial \w{\hom} commutes
with filtered (homotopy) colimits and coproducts.

By the discussion of Example \ref{exam2.2}, we can identify the localization map
\w{B(C\sp{f}) \to \LLc\sb{B(\W\sp{f})}(B(C\sp{f}))} with homotopy coherent nerve of
the inclusion \w{C\sp{f} = C\sp{\circ} \to \mC\sp{\circ}} of $C$ as a discrete
simplicial category; note that by \cite[Remark A.3.2.17]{LurieHTT}, every object
in an excellent model category is cofibrant. The subcategory $\W$ is saturated,
since a map in \w{C\sp{f}} is a weak equivalence if and only if it represents
an equivalence in \w[.]{\fB(\mC\sp{\circ})}

The quasi-category \w{B(C\sp{f})} presents the trivial model structure on
\w[,]{C\sp{f}} so Assumption (2) is equivalent to showing that for some sufficiently
large regular cardinal $\lambda$, $\lambda$-filtered colimits are homotopy colimits.
But this is true for all combinatorial model categories
(see \cite[Proposition 7.3]{DuggC}, and thus all excellent model categories.
\end{example}

\begin{defn}\label{def3.3}
An object \w{x \in X} is \emph{$K$-local} if it is in the essential image of the
right adjoint of the localization map \w[.]{\LWX\to\LK\sp{\cocon}(X)}
\end{defn}

\begin{remark}\label{rmk3.4}
By definition $x$ is \ww{K}-fibrant if it is \ww{K}-local and fibrant.

As in the model category case (see \cite{BarwL}), the $K$-local objects are the
objects $z$ such that the $K$-equivalences \w{f:x\to y} induce bijections
\w[.]{\Map\sb{\LWX}(z,x)\to \Map\sb{\LWX}(z,y)}
\end{remark}

\begin{lemma}\label{lemexact}
The map \w{\LWX\to\LK\sp{\cocon}(X)} is left exact (i.e., preserves pullbacks).
\end{lemma}

\begin{proof}
By \cite[Proposition 6.2.1.1]{LurieHTT}, and the preceding paragraph, it suffices
to show that the $K$-equivalences are stable under pullback. But this follows
from the fact that the functor \w{\Map\sb{\LWX}(h, -)} preserves limits.
\end{proof}

\begin{lemma}\label{lem3.7}
A $K$-equivalence between $K$-local objects is a weak equivalence.
\end{lemma}

\begin{proof}
Suppose that \w{f : x \to y} is a $K$-equivalence. Then by Definition \ref{con2.5}
the image of $f$ in \w{\LWX} is in the essential image of
\w[,]{\J(\LK\sp{\cocon}(X)) \subseteq\LK\sp{\cocon}(X) \xrightarrow{\phi} \LWX}
where $\phi$ is the right adjoint of localization. In particular it is an equivalence
in \w[,]{\LWX} so that $f$ is a weak equivalence by Assumption \ref{ass3.1}(3).
\end{proof}

\begin{construction}\label{con2.8}
By \cite[Proposition 5.5.4.15]{LurieHTT}, the $K$-equivalences are the saturation of
the set $K$. Thus, by a small object argument of sufficient size, we can
construct a fibrant replacement \w{x\to\M\sb{K}(x)} of an object by a $K$-local one.

Consider the localization map
\w[.]{i\sb{\ast}:\LWX \leftrightarrows \LK \sp{\cocon}(X) : i\sp{\ast}}
  There is a commutative diagram in \w[:]{\LWX} 
$$
\xymatrix
{
x \ar[r] \ar[d] & \M\sb{K}(x) \ar[d] \\
i\sp{\ast}i\sb{\ast}(x) \ar[r] & i\sp{\ast}i\sb{\ast}\M\sb{K}(x)
}
$$
\noindent where the vertical maps are the counits of the adjunction.
The bottom horizontal and right vertical maps are equivalences in \w{\LWX}
by Lemma \ref{lem3.7} above.
\end{construction}

\begin{construction}\label{con3.2}
We endow $X$ with the structure of a category of fibrations and weak equivalences
(called the \emph{\ww{K}-equivalence structure}), as follows: the fibrations between
fibrant objects of $X$ are defined to be those morphisms of \w{\Fib} for which the
diagram
$$
\xymatrix
{
x \ar[d] \ar[r] & \mM\sb{K}(x) \ar[d] \\
y \ar[r] &  \mM\sb{K}(y)
}
$$
is a pullback in \w[.]{\LK(X)}

As usual, a map which is both a $K$-equivalence and a $K$-fibration is called
a $K$-\emph{trivial fibration}.
\end{construction}

\begin{thm}\label{thm3.5}
For each \w[,]{m \in \NN} the $K$-fibrations and $K$-equivalences
endow $X$ with the structure of a quasi-category with fibrations and weak equivalences.
\end{thm}

To prove this, we shall require some preliminary results:

\begin{lemma}\label{lem3.8}
Suppose we have a pullback diagram
$$
\xymatrix
{
x \times\sb{z} y \ar[r] \ar[d]\sb{g} & y \ar[d]\sp{f} \\
x \ar[r] & z
}
$$
\noindent where $f$ is a $K$-fibration and $z$, $x$, and $y$ are fibrant.
Then $g$ is a $K$-fibration as well.
 \end{lemma}

\begin{proof}

Consider the diagram:
$$
\xymatrix@C=1.5em@R=1.5em{
& x \times\sb{z} y \ar@{-}[rr] \ar@{-}[d]_>>>{g} \ar@{-}[dd]
& & x \ar@{-}[dd]
\\
\mM\sb{K}(x \times\sb{z} y) \ar@{-}[ur]\ar@{-}[rr]\ar@{-}[dd]
& & \mM\sb{K}(x) \ar@{-}[ur]\ar@{-}[dd]
\\
& y \ar@{-}[rr]
& & z
\\
\mM\sb{K}(y) \ar@{-}[rr]\ar@{-}[ur]
& &\mM\sb{K}(z) \ar@{-}[ur]
}
$$
in \w[.]{\LWX} Since localization preserves pullbacks by fibrations by
\cite[Theorem 7.5.18]{CisiH}, we can assume that the back face is a pullback in
\w[.]{\LWX}

The front face of the cube is equivalent in \w{\LWX} to \w{i\sp{\ast}i\sb{\ast}}
applied to the back face by the discussion of \S \ref{con2.8}.
Thus, the front face of the cube is a pullback in \w{\LWX} by
Lemma \ref{lemexact}. The images of the back and front faces of the cube in
\w{\LLc \sb{K}(X)} are also pullbacks by another application of
Lemma \ref{lemexact}. By hypothesis, the right face is a pullback in
\w{\LWX} so the pasting law for pullbacks in a
quasi-category (\cite[Lemma 4.4.2.1]{LurieHTT}) implies the required result.
\end{proof}

\begin{lemma}\label{lem3.9}
Suppose we have a pullback diagram in $X$
$$
\xymatrix
{
x \times\sb{z} y \ar[r] \ar[d]\sb{h} & y \ar[d]\sp{f} \\
x \ar[r]\sb{g} & z
}
$$
where $f$ is a $K$-fibration, $z, x, y$ are fibrant and $g$ is a
$K$-equivalence. Then $h$ is a $K$-equivalence as well.
\end{lemma}

\begin{proof}
The image of the pullback square in $X$ in \w{\LWX} is a pullback as well by
\cite[Theorem 7.5.18]{CisiH}. Pullbacks preserve $K$-equivalences in
\w[,]{\LWX} by \ref{lemexact}.
\end{proof}

\begin{prop}\label{lem3.10}
Suppose that \w{f : x \to z} is a map between $K$-fibrant objects.
Then we can factor it as a $K$-equivalence followed by a $K$-fibration.
\end{prop}

\begin{proof}

Form the diagram
$$
\xymatrix
{
x \ar[ddr]\sb{\psi} \ar[drr] \ar@{.>}[dr] & &  \\
& \ar[d] y'' \ar[r] & \mM\sb{K}(x) \ar[d] \\
 & y' \ar[r] \ar[d]\sb{\phi} & y \ar[d] \\
 & z \ar[r] & \mM\sb{K}(z)
}
$$
\noindent in $X$, where the right vertical composite is a factorization of
\w{\mM\sb{K}(f)} as an equivalence followed by a fibration,
and both squares are pullbacks.
Note that we can assume that \w{\mM\sb{K}(z)} and \w{\mM\sb{K}(x)}
are fibrant, so that the above pullbacks are guaranteed to exist.
We claim that \w{\phi \circ \psi} gives the required factorization:

The horizontal map \w{y' \to y} is a $K$-equivalence by Lemma \ref{lem3.9}
above. The objects \w[,]{\mM\sb{K}(y)} \w[,]{\mM\sb{K}(x)} and \w{y'}
are  $K$-local, so \w{\mM\sb{K}(x)\to y}  is a
weak equivalence, and hence a $K$-equivalence.
By the $2$-out-of-$3$ property, we see that $\psi$ is a $K$-equivalence.

To check that $\phi$ is a $K$-fibration, it suffices to show that
$y$ is weakly equivalent to $\mM\sb{K}(y')$. Since $\mM\sb{K}(z)$
is $K$-fibrant, we can conclude that $y$ is $K$-fibrant. A
$K$-equivalence between $K$-local objects is a weak equivalence by
Lemma \ref{lem3.7}. Thus $y$ is $K$-local, being weakly equivalent to
\w[.]{\mM\sb{K}(x)} It follows from Remark \ref{rmk3.4} that $y$ is
$K$-fibrant. It is also $K$-equivalent, and hence
weakly equivalent, to \w[.]{\mM\sb{K}(y')}
\end{proof}

\begin{proof}[Proof of Theorem \protect{\ref{thm3.5}}]
First, we check that the $K$-fibrations form a class of fibrations
in the sense of \ref{def1.1}. Property (3) is just a special case of \ref{lem3.8}
above. It is immediate from
the definition that the identity map is a $K$-fibration.
Pullbacks of $K$-fibrant objects exist, since they are in particular
pullbacks of fibrant objects. Composition preserves $K$-fibrations by the
2-out-of-3 property for pullbacks in a quasi-category (\cite[Lemma 4.4.2.1]{LurieHTT}).

Pullbacks of fibrations in $X$ are taken to pullbacks in
\w{\LWX} by \cite[Theorem 7.5.18]{CisiH}.  If $f$ is a $K$-equivalence and
a $K$-fibration, the induced map \w{\mM\sb{K}(x) \rightarrow \mM\sb{K}(y)} is an
equivalence in \w[.]{\LWX} Thus, $f$ is a pullback of an equivalence in \w[,]{\LWX}
and thus itself represents an equivalence in $\LWX$. It is also a $K$-fibration by \ref{lem3.8} above. 

\end{proof}

\begin{remark}\label{rmk3.11}
Given a model category $\bM$, the fibrations of a left Bousfield localization
\w{\LLc(\bM)} of $\bM$ are somewhat difficult to describe.
The paper \cite{BarwL} gives a nice characterization of fibrations whose target
is in an admissible left exact and right proper subcategory $E$
(see \cite[Definition 4.15]{BarwL}). An example of such a subcategory
is the subcategory of fibrant objects (see \cite[Example 4.17]{BarwL}).
We used this characterization of fibrations in the localization with
fibrant source as our definition of $K$-fibrations above.
\end{remark}

%
%
\sect{Localization with respect to a set of objects}
\label{clocth}

In this section, we will show that the localization of a locally presentable
quasi-category with respect to a class of $H$-equivalences (associated to a small
set of objects $H$) is a right Bousfield localization in the sense of
Definition \ref{def2.3}.
This is an application of the Adjoint Functor Theorem of \cite{Raptis}, using
the fact that the mapping spaces of the localization are locally small, by
Lemma \ref{lem2.6} below. The proof of this Lemma is inspired by the
``bounded cofibration arguments'' common in localization theory for model categories
(see the introduction of \cite{JardC} as well as Lemma 4.9 there).

\begin{defn}\label{def2.4}
Let \w{(X,\Fib,\W)} be a quasi-category with fibrations and weak equivalences.
Let $H$ be a small collection of objects in $X$. We say that a $1$-morphism
\w{f : x \to y} in $X$ is an \emph{$H$-equivalence} if
$$
\Map\sb{\LWX}(h, x) \to \Map\sb{\LWX}(h, y)
$$
is a weak equivalence for each \w[.]{h \in H} An object $z$ is called \emph{$H$-local}
if and only if, for each $H$-equivalence \w[,]{f : x \to y} the induced map
$$
\Map\sb{\LWX}(y, z) \to \Map\sb{\LWX}(x, z)
$$
is a weak equivalence.
\end{defn}

Note that every map in $\W$ is automatically an $H$-equivalence, since it is
an equivalence in \w[.]{\LWX} By a slight abuse of notation, we will write
\w{\LHX} for the localization of $X$ at the $H$-equivalences.

\begin{lemma}\label{lem2.5}
Suppose that $X$ is a locally presentable quasi-category. Then
\begin{enumerate}
\renewcommand{\labelenumi}{(\arabic{enumi})~}
\item Each object is $\lambda$-compact for some regular cardinal $\lambda$.
\item For each regular cardinal $\lambda$, the set of $\lambda$-compact objects
  is essentially small.
\end{enumerate}
\end{lemma}

\begin{proof}
(1)\ By \cite[Proposition 5.4.2.2]{LurieHTT}, we can present any \w{x\in X} as a
$\kappa$-filtered colimit of $\kappa$-compact objects, indexed by a diagram $I$.
Since we can raise the index of accessibility (\cite[Proposition 5.4.2.11]{LurieHTT}),
we can assume that $X$ is a $\lambda$-accessible for \w[.]{\lambda >|I|} Thus, $x$ is a
colimit of a $\lambda$-bounded diagram of $\lambda$-compact objects and is thus
$\lambda$-compact by \cite[Corollary 5.3.4.15]{LurieHTT}.

\noindent(2)\ Using \cite[Proposition 5.4.2.11]{LurieHTT}, choose \w{\lambda' > \lambda}
such that $X$ is \ww{\lambda'}-accessible. By \cite[Proposition 5.4.2.2]{LurieHTT},
the set of \ww{\lambda'}-compact, and hence $\lambda$-compact, objects of $X$ is
essentially small.
\end{proof}

\begin{lemma}\label{lem2.6}
Consider a locally presentable quasi-category $X$ equipped with fibrations \w{\Fib}
and weak equivalences $\W$.  Suppose that \w{\LWX} is locally presentable accessible
localization of $X$. Then for each \w{x\in X} there is a regular cardinal $\lambda$
such that for each $H$-equivalence \w{s : y \to x} there is a $\lambda$-compact object
$z$ in \w{\LWX} and a diagram of $H$-equivalences
$$
\xymatrix
{
z \ar[d] \ar[r] &x \\
y \ar[ur]\sb{s} &
}
$$
\end{lemma}

\begin{proof}
Let $S$ be a small set of objects such that each element of $X$ can be written
canonically as a filtered colimit of elements of $S$
(see \cite[Proposition 5.5.1.1]{LurieHTT}). By Lemma \ref{lem2.5} above, we can choose a
regular cardinal $\lambda$ such that
\begin{enumerate}
\renewcommand{\labelenumi}{(\arabic{enumi})~}
\item $H$ is $\lambda$-bounded, and for each \w{h \in H} the functor
  \w{\Map\sb{\LWX}(h, -)} commutes with $\lambda$-filtered colimits.
\item For each choice of \w{a,b \in (H \cup S)} and \w[,]{n \in \NN} the set of
$n$-simplices of \w{\Map\sb{\LWX}(a,b)} is $\lambda$-bounded.
\item The localization map is $\lambda$-accessible: that is, it preserves
$\lambda$-filtered colimits.
\end{enumerate}

Let us write
$$
y~=~\underset{s \in I}{\colim} \,y\sb{s}
$$
as a $\lambda$-filtered colimit of objects in $S$. For each \w[,]{\lambda' < \lambda}
we define a sub-object \w{z\sb{\lambda'}} of $y$ by transfinite induction.
We start the induction as follows: for each object \w{h \in H} and
\w[,]{\alpha \in \pi\sb{i}\Map\sb{X}(h, x))} we choose an object
\w{y\sb{(i, \alpha, h)} := y\sb{s}} \wb{s \in I} such that
$$
\alpha \in\Image(\pi\sb{i}\Map\sb{\LWX}(h, y\sb{(i, \alpha, h)})\to
\pi\sb{i}\Map\sb{\LWX}(h, x)).
$$
\noindent Let \w{I\sb{0} \subseteq I} be the full subcategory with objects of the form
\w[.]{y\sb{(i, \alpha, h)}} This is guaranteed to have a $\lambda$-bounded set of
morphisms by assumption (2) on $\lambda$. Now put
$$
z\sb{0}~:=~\underset{s \in I\sb{0}}{\colim} \, y\sb{s}.
$$
In the inductive step there are two possibilities\vsm:

\noindent\textbf{Case 1:} for a successor cardinal \w[,]{\lambda''+1=\lambda'} let
$$
z\sb{\lambda''}~:=~\underset{s \in I\sb{\lambda''}}{\colim} \, y\sb{s}
$$
\noindent for \w[.]{I\sb{\lambda''} \subseteq I} Given \w{k\in\NN} and two elements
$\alpha$ and $\beta$ in \w{\pi\sb{k}\Map\sb{X}(h,z\sb{\lambda''})} with the same image
in \w[,]{\pi\sb{k}\Map\sb{\LWX}(h, x)} choose for each
\w{i\in I\sb{\lambda''}} a commutative diagram
$$
\xymatrix{
  \pi\sb{k}\Map\sb{\LWX}(h, y\sb{j}) \ar[r] \ar[d]\sb{\phi} &
  \pi\sb{k}\Map\sb{\LWX}(h, y\sb{(i, \alpha, \beta)}) \ar[r] &
  \pi\sb{k}\Map\sb{\LWX}(h, x) \\
  \pi\sb{k}\Map\sb{X}(h, z\sb{\lambda''})  & &
 }
$$
with \w[,]{j \in I\sb{\lambda''}} \w[,]{\phi(\alpha')=\alpha}
and \w{\phi(\beta') = \beta} for some \w{\alpha'} and \w{\beta'} whose images under the
left horizontal map are the same.

We then let \w{I\sb{\lambda'}} be the full subcategory of $I$ on objects
in \w{I\sb{\lambda''}} and those of the form \w{y\sb{(i, \alpha, \beta)}} above.
The colimit
$$
y = \underset{s \in I\sb{\lambda'}}{\colim} \, y\sb{s}
$$
is evidently $\lambda$-compact, since the set of elements \w{y\sb{(i, \alpha, \beta)}}
is $\lambda$-bounded\vsm.

\textbf{Case 2:} if \w{\lambda'} is a limit cardinal, put
\w[.]{z\sb{\lambda'}:=\underset{\lambda''<\lambda'}{\colim}\,z\sb{\lambda''}}

At each stage of the induction, we obtain a $\lambda$-compact object, and the object
$$
z~=~\underset{\lambda' < \lambda}{\colim}\,z\sb{\lambda'}
$$
is also $\lambda$-compact. Both of these statements follow from
\cite[Corollary 5.3.4.15]{LurieHTT}. It is easy to check that the canonical map
\w{z \to y} is an $H$-equivalence by the construction of \w{z\sb{\lambda'}}
and the fact that \w{\Map\sb{\LWX}(h, -)} commutes with $\lambda$-filtered colimits.
\end{proof}

\begin{thm}\label{thm2.7}
Given $X$ and \w{\LWX} as in Lemma \ref{lem2.6}, let $H$ be a small set of
objects of $X$ whose image in \w{\LWX} is connected. Then the localization map
\w{\LWX\to\LHX} is a right Bousfield localization.
\end{thm}

\begin{proof}
By \cite[Proposition 7.11.2]{CisiH}, it suffices to show that the localization map
has a right adjoint. By the version of the adjoint functor theorem from (\cite{Raptis}),
it suffices to show that this localization map is continuous, that \w{\LHX} is
locally small, and that it satisfies the solution set condition.

If we equip \w{\LWX} with the class of fibrations and weak
equivalences given by the dual of Example \ref{exam1.3}, then \w{\LHX} is exactly
the continuous localization of \w{\LWX} (in the sense of \cite[Remark 7.7.10]{CisiH}).
This follows from the fact that both pullbacks and products of $H$-equivalences are
$H$-equivalences.

To show this, by \cite[Proposition 7.10.1 \& Corollary 7.6.13]{CisiH} it suffices
to show that \w{\PPc(\LHX)} is locally small. In other words, we want to show that
\w{\pi\sb{0}\Map\sb{\LHX}(x, y)} is small for each \w[.]{x, y \in X}
The description of the mapping space obtained from the dual of
\cite[7.2.10(2), Remark 7.2.21]{CisiH} (in terms of the calculus of
fractions), together with Lemma \ref{lem2.5}, imply that each component
of \w{\Map\sb{\LWX}(x,y)} contains an object of the form \w[,]{x \leftarrow z \to y}
where $z$ is $\lambda$-compact in $\LLc (X)$ for some cardinal $\lambda$. But the
$\lambda$-compact objects of $X$ are essentially small. Thus the set of such
components is small.

We now verify the solution set condition from \cite{Raptis}. That is, we want to show
that \w{\LHX\sb{y/}} has a small, weakly initial set. Every object of \w{\LHX\sb{y/}}
admits a morphism  from an object \w{y \leftarrow z \rightarrow x} such that $z$ is
$\alpha$-bounded, as noted in the preceding paragraph.
On the other hand, every such morphism factors through \w[,]{y \leftarrow z\to x'}
where \w{x'} is $\alpha$-bounded. The set of all morphisms \w[,]{y \leftarrow z\to x'}
for $z$ and \w{x'} $\alpha$-bounded is essentially small.
Hence the solution set condition holds.
\end{proof}

%
%
\sect{Right Bousfield Localization}
\label{crbl}

In this section, we construct a version of right Bousfield localization with respect
to a set of objects $H$, both for a quasi-category equipped with cofibrations
(see Definition \ref{dqccw} below) and one equipped with fibrations.
This is not formally dual to the case of left Bousfield localization above,
since we are localizing with respect to a set of \emph{objects}, not morphisms.

\begin{defn} \label{dqccw}
A triple \w{(X, \Cof,\W)} is called a \emph{quasi-category with
  cofibrations and weak equivalences} if \w{(X\op, \Cof\op,\W\op)}
is a category with fibrations and weak equivalences as in Definition \ref{def1.2}.
\end{defn}

Our first example is obtained by dualizing Corollary \ref{cor1.6}:

\begin{lemma}\label{cor1.8}
Suppose that $X$ has countable colimits. Then \w{\cpX} (\S \ref{snac}) can be equipped
with the structure of a quasi-category with weak equivalences and cofibrations dual to
that of Theorem \ref{thm1.5}. We call this the \emph{Reedy structure} on \w[.]{\cpX}
\end{lemma}

In this section we shall assume that $X$ is locally
presentable, and fix a collection $H$ of cofibrant objects in $X$, each of
which is compact and connected as an object of \w[.]{\LWX} We wish to endow $X$
with a new structure of a category of cofibrations and weak equivalences, in which the
weak equivalences are the $H$-equivalences; we do so by mimicking Barwick's
construction of right Bousfield localizations in \cite{BarwL}.

\begin{assume}\label{ass6.1}
We assume for simplicity a few additional properties for \w[:]{(X,\Cof,\W)}
\begin{enumerate}
\renewcommand{\labelenumi}{(\arabic{enumi})~}
\item \w{X \to \LWX} preserves transfinite composites of cofibrations of
 cofibrant objects.
\item The structure is cofibrantly generated in that each (trivial) cofibration
  can be written as a transfinite composite of pushouts of a small set of
  (trivial) cofibrations with cofibrant domain.
\item The set of weak equivalences satisfies (3) of \ref{ass3.1}
\end{enumerate}
\end{assume}

These are fairly mild assumptions. For instance, if we look at the
structure given by \ref{exam1.4} on the nerve of the cofibrant objects of a model
category assumption (1) is satisfied by \cite[Proposition 17.9.1]{HirM}.
Property (2) and (3) are enjoyed, for instance by the structure from \S \ref{examxxx},
as an excellent model category is cofibrantly generated by definition.

For each \w[,]{h \in H} let \w{\Lu(h)} be a cofibrant replacement for the constant
cosimplicial object on $h$ (in the Reedy structure on \w[),]{\cpX} and let
$$
I\sb{H}~:=~\{ L\sb{p}\Lu(h) \to\Lambda\sp{p}(h)\}\sb{h\in H,p\in\NN} \cup J~,
$$
where $J$ is a set of generating trivial cofibrations for \w{(X,\Cof,\W)}
with cofibrant domains, and \w{L\sb{p}} is the $p$-th latching object (see
\cite[VII, \S 4]{GJarS}).

\begin{defn}\label{def6.2}
A map in $X$ is called an \emph{$H$-cofibration} if it can be written as a transfinite
composite of pushouts of elements of \w[.]{I\sb{H}} We call a map an
\emph{$H$-trivial cofibration} if it is both an $H$-cofibration and an $H$-equivalence.
\end{defn}

\begin{defn}\label{dhcol}
We call an object \w{x \in X} \emph{$H$-colocal} if for each $H$-equivalence
\w{f : y \to z} the induced map \w{\Map\sb{\LWX}(x, y)\to\Map\sb{\LWX}(x,z)}
is an equivalence.
\end{defn}

\begin{lemma}\label{lem6.4}
$H$-equivalences between $H$-colocal objects in $X$ are weak equivalences.
\end{lemma}

\begin{proof}
The $H$-equivalences between $H$-colocal objects are in the essential image of
\w[,]{\J\LH(X) \subseteq\LH(X) \xrightarrow{\phi} \LWX} where $\phi$ is
the left adjoint of the localization map, whose existence is guaranteed by
Theorem \ref{thm2.7}. Thus, $H$-equivalences between $H$-local objects represent
equivalences in \w[.]{\LWX} They are therefore weak equivalences by
Assumption \ref{ass6.1}.
\end{proof}

\begin{lemma}\label{lem6.5}
Suppose that $x$ is $H$-cofibrant. Then a map \w{f : x \to y} in $X$ is an $H$-trivial cofibration
if and only if it is a trivial cofibration.
\end{lemma}

\begin{proof}
If $f$ is a trivial cofibration it is an $H$-cofibration, since
\w[.]{J \subseteq I\sb{H}} It is also an $H$-equivalence, since every weak equivalence
is an $H$-equivalence.

Conversely, suppose that \w{f : x \to y} is an $H$-trivial cofibration.
Then both objects $x$ and $y$ can be written as colimits of transfinite composites of
pushouts of maps in \w[,]{I\sb{H}} and thus transfinite composites of pushouts of
$H$-colocal maps. We claim that $x$ and $y$ are thus $H$-colocal.

Indeed, the functorial mapping space \w{\MAP\sb{\LWX}(-, y)} sends colimits to limits
(see the discussion in the preceding section). Thus, colimits preserve $H$-colocal
objects of \w[.]{\LWX} But the localization \w{X \to \LWX} preserves pushouts
by cofibrations by the dual of \cite[Theorem 7.5.18]{CisiH} and also preserves
transfinite composites of cofibrations by (1) of \ref{ass6.1}. Thus,
$x$ and $y$ are $H$-colocal.

But $H$-equivalences between $H$-colocal objects are
weak equivalences, so $f$ is a weak equivalence. The elements of \w{I\sb{H}} are
all cofibrations, so we conclude that an $H$-cofibration is a cofibration.
\end{proof}

\begin{lemma}\label{lem6.6}
Let \w{f : x \to y} be a map with $H$-cofibrant source. Then we can factor it as
an $H$-cofibration followed by an $H$-equivalence.
\end{lemma}

\begin{proof}
By Lemma \ref{lem2.5}, each object of $H$ is $\lambda$-compact for some cardinal
$\lambda$.
By a small object argument of size \w{\lambda' > \lambda} for some regular
\w[,]{\lambda'} we can factor $f$ as \w[,]{x \xrightarrow{g} z \xrightarrow{g'} y}
where $g$ is an H-cofibration, and the map \w{g'} has the right
lifting property with respect to \w[.]{h \otimes\partial\Deln{n} \to h\otimes\Deln{n}}

We now want to show that the map $g'$ induces an $H$-equivalence.
To do this, we will show that we can solve all lifting problems
\mydiagram[\label{eqliftsq}]{
\partial\Deln{n} \otimes h \ar[r]\sp(0.5){q} \ar[d] & X \ar[d]\sp{g'} \\
\Deln{n} \otimes h \ar[r]\sb(0.5){r} \ar@{.>}[ur] & Y
}
in \w[.]{\LWX} Note that the construction only shows that we can solve such lifting
problems in $X$.

We can write the horizontal maps in \wref{eqliftsq} as composites \w{q'\circ q''} and
\w[,]{r' \circ r''} where \w{q'} and \w{r'} are in the image of the localization
map \w{X\to\LWX} (see \cite[\S 7.2]{CisiH}), so it suffices to solve a lifting problem:
$$
\xymatrix
{
\partial\Deln{n} \otimes h \ar[r]_{q''} \ar[d] & w \ar[r]_{q'} &  X \ar[d]_{g'} \\
\Deln{n} \otimes h \ar[r] \ar[r]_{r''} & w \ar@{.>}[ur]  \ar[r]_{r'} & Y.
}
$$
\noindent Without loss of generality, we can assume \w{q'} and \w{r'} have
$H$-cofibrant domains, so we reduce to solving lifting problems involving elements
of \w[,]{I\sb{H}} which we can solve by the construction of \w[.]{g'}

For each \w{h \in H} and \w[,]{\sigma \in \Omega\sp{n}\Map\sb{\LW (X)}(h, z)}
we can find lifts in the diagrams (by the preceding discussion)
$$
\xymatrix
{
h \otimes \partial\Deln{n} \ar[r]\sb{0} \ar[d] & y \ar[d]\sb{g'} \\
h \otimes\Deln{n} \ar@{.>}[ur] \ar[r]\sb{\sigma} & z
}
$$
On the other hand, given \w{\sigma'  \in \Omega\sp{n}\Map\sb{\LWX}(h, z)}
with a nullhomotopy of \w{g'\circ\sigma'} given by $\gamma$ below, we can find a lift in
the diagram
$$
\xymatrix
{
h \otimes \partial \Deln{n+1} \ar[r]\sb{\sigma'} \ar[d] & y \ar[d]\sp{g'} \\
h \otimes \Deln{n+1} \ar[r]\sb{\gamma} \ar@{.>}[ur] & z.
}
$$
\noindent We conclude that
\w{\pi\sb{0}\Omega\sp{n}\Map\sb{\LWX}(h, z)\to\pi\sb{0}\Omega\sp{n}\Map\sb{\LWX}(h, y)}
is a bijection.
\end{proof}

\begin{thm}\label{thm6.7}
The $H$-equivalences and $H$-cofibrations give $X$ the structure of a quasi-category
with weak equivalences and cofibrations
\end{thm}

\begin{proof}
The $H$-cofibrations form a class of cofibrations by a standard argument
(see \cite[Proposition 5.6]{BarwL}). The dual of (2) of \ref{def1.2} is just
Lemma \ref{lem6.6}. The dual of Condition (1) of \ref{def1.2} follows from the fact that
the $H$-trivial cofibrations with $H$-cofibrant source are precisely the trivial
cofibrations, and that pushouts of trivial cofibrations are trivial cofibrations.
\end{proof}

We also have a version of right Bousfield localization with respect to a class of
objects in a quasi-category with fibrations, which is much easier to establish:

\begin{thm}\label{thm4.10}
Suppose that \w{(X,\Fib,\W)} is a quasi-category with fibrations. Then we can equip $X$
with the structure of a quasi-category with fibrations and weak equivalences in which
the weak equivalences are the $H$-equivalences and the fibrations are those
of \w[.]{\Fib}
\end{thm}

\begin{proof}
Mapping space preserve limits, so $H$-equivalences are preserved under pullback.
Thus (1) of \ref{def1.2} holds. (2) is immediate from the factorization axiom for
\w[,]{(X,\W,\Fib)} since a weak equivalence is in particular an $H$-equivalence.
\end{proof}

%
%
\sect{The spectral sequence of a simplicial object}
\label{csssso}

Assume given a complete quasi-category with fibrations and weak equivalences
\w{\lra{X,\Fib,\W}} as in Definition \ref{def1.2}, satisfying Assumptions \ref{ass3.1}.
Given a homotopy cogroup object $\fy$ in $X$ (or more generally, in any
suitable version of an \wwb{\infty,1}category), we explained in \cite{BMeadS} how to
associate to a simplicial object \w{\xd} in $X$ its \emph{homotopy spectral sequence},
and provided a homotopy-invariant characterization of the differentials,
independent of the model of $\infty$-categories chosen.  For this purpose, we first
require:

\begin{lemma}\label{lem8.2}
There is an adjunction of quasi-categories
$$
\fy\otimes (-):\Ss \leftrightarrows X :\MAP\sb{X}(\fy, -)~,
$$
\noindent where the tensoring $\otimes$ of a locally presentable quasi-category
over simplicial sets is that given in \cite[Section 4.4.4]{LurieHTT}.
\end{lemma}

\begin{proof}
By \cite[Proposition A.3.7.6]{LurieHTT}, \w[,]{X \simeq \fB(\mA)} where $\mA$ is
the underlying simplicial category of a combinatorial simplicial model category $\bA$.
There is a Quillen adjunction
$$
\fy\otimes (-):\leftrightarrows :\map\sb{\mA}(\fy, -)~,
$$
where $\otimes$ comes from the simplicial structure of $\bA$.
By \cite[Proposition 5.2.4.6]{LurieHTT}, this induces an adjunction of quasi-categories.
\end{proof}

\begin{mysubsection}{The spectral sequence of a simplicial object}
\label{sssso}
We briefly recall the necessary background from \cite{BMeadS}, in the case
where $X$ is a quasi-category (with enough colimits and limits), $\fy$ is a
compact homotopy cogroup object of \w[,]{\LWX} and \w{\xd\in sX} is a simplicial
object in $X$. The associated spectral sequence, originally due to Quillen in
\cite{QuiS} (see also \cite{BFrieH}), has the form:
\begin{myeq}\label{eqssso}
E\sp{1}\sb{n,p}~=~\pi\sb{p}\Map\sb{X}(\fy, x\sb{n})~
~\implies~\pi\sb{p+n}\Map\sb{X}(\fy,\|\xd\|)~,
\end{myeq}
\noindent using the mapping spaces of Lemma \ref{lem8.2}.

The spectral sequence  is in fact determined by the restriction of \w{\xd} to \w{\spX}
(see \S \ref{snac}). This is also true of \w[,]{\|\xd\|} the
\emph{geometric realization} (colimit) of \w{\xd} (see \cite[Appendix A]{SegCC}).
Therefore, from now on we shall work with restricted simplicial objects.

With some exceptions (see, e.g., \cite[\S 4]{BKanSQ}) the only useable spectral
sequences are those of abelian groups, so it is not a significant restriction
to assume that \w{\fy=\Sigma\hy} is in fact a suspension in $X$.

In distinction from the usual approach using exact couples
\w[,]{(E\sp{r},D\sp{r})} for our purposes we work throughout by
representing a class $\gamma$ in \w{E\sp{r}\sb{n,p}} of the
spectral sequence by elements
\w{[f]\in\pi\sb{p}\Map\sb{X}(\hy,x\sb{n})} in
\w[.]{E\sp{1}\sb{n,p}} Each such \w{[f]} is a homotopy invariant
in $X$, or in the Reedy model structure on \w[,]{\spX} but of
course we may have many such representatives for the given
\w[.]{\gamma=\lra{f}} From the general theory we know that the
differential \w{d\sp{r}(\gamma)} vanishes if and only if there
exists a representative \w{[f]} for which \w{\fd\sp{r}[f]=0} in
\w[,]{E\sp{1}\sb{n-r,p+r-1}} where  \w{\fd\sp{r}[f]} is defined by
a particular choice of lifts in terms of the original exact couple
\w[.]{(E\sp{1},D\sp{1})}

Specializing to the Bousfield-Friedlander spectral sequence of a simplicial space
(in the version of \cite[\S 8]{DKStovB}), we showed in \cite[Theorem 3.11]{BMeadS}
that if \w{\gamma\in E\sp{r}\sb{n,p}} is represented by a map
\w[,]{f:\Sigma\sp{p}\hy\to x\sb{n}} then \w{[f]} survives to \w{E\sp{r}\sb{n,p}}
(for \w[)]{r\geq 2} if and only if we can fit $f$ into a diagram in $X$ of the form:
\myvdiag[\label{diagnk}]{
  \Sigma\sp{p}\hy  \ar@/^1pc/[rr]\sp{d\sb{0}=0}\sb{\vdots}
  \ar@/_1pc/[rr]\sb{d\sb{n}=0} \ar[dd]\sb{f}  && 0
  \ar@/^1pc/[rr]\sp{d\sb{0}=0}\sb{\vdots}  \ar@/_1pc/[rr]\sb{d\sb{n-1}=0} \ar[dd] &&
  0 \ar[dd] &\cdots\cdots & 0 \ar[dd] \\
  && && && \\
x\sb{n} \ar@/^1pc/[rr]\sp{d\sb{0}}\sb{\vdots}  \ar@/_1pc/[rr]\sb{d\sb{n}} && x\sb{n-1}
\ar@/^1pc/[rr]\sp{d\sb{0}}\sb{\vdots} \ar@/_1pc/[rr]\sb{d\sb{n-1}} && x\sb{n-2} &
\cdots \cdots & x\sb{n-r+1}~.
}
\noindent Moreover, the value of the differential \w{\fd\sp{r}([f])} is represented in
\w{E\sp{1}\sb{n-r,p+r-1}} by a map \w{\Sigma\sp{p+r-1}\hy\to x\sb{n-r}} constructed by
the universal property of the diagram of the form \wref[,]{diagnk}
(see \cite[Corollary 6.11]{BMeadS}), so the fact that
\w{\fd\sp{r}([f])} is zero in \w{E\sp{r}\sb{n-r,p+r-1}} \wwh in other words, having
\emph{some} nullhomotopic representative in \w[,]{E\sp{1}\sb{n-r,p+r-1}} for some choice
of such an extension \wh is equivalent to \w{\lra{f}} surviving to
\w[.]{E\sp{r+1}\sb{n,p}} See \S \ref{sdiagdesc} below for an alternative construction
of \w[.]{\fd\sp{r}([f])}
\end{mysubsection}

\begin{remark}\label{rerterm}
In order to calculate the usual \ww{E\sp{r+1}\sb{n,p}}-term of a spectral sequence
described as above using the \ww{E\sp{1}}-term alone, we need to know not only which
classes \w{[f]\in E\sp{1}\sb{n,p}} survive to the \ww{E\sp{r+1}}-term, but also which
of them are hit by the \ww{d\sp{m}}-differential \wh or more precisely, are in the
image of \w{\fd\sp{m}} \wwh for some \w[.]{1\leq m\leq r}
This involves an analysis of all possible diagrams of the form \wref{diagnk} for
the given values of \w[,]{(n,r,p)} as well as those for \w{(n+m-1,m,p-m+1)}
with \w[.]{p\leq m\leq r}

However, when \w[,]{n<r} diagrams of the form \wref{diagnk} do not exist, but we must
take into account all diagrams starting in dimension $n$ and terminating in
\w{x\sb{1}} used to calculate \w{\fd\sp{n}} itself in order to know if
\w{[f]\in E\sp{1}\sb{n,p}} survives to the \ww{E\sp{n+1}}-term, and thus by default to
the \ww{E\sp{r+1}}-term.
\end{remark}

\begin{mysubsection}{Chain complexes}
\label{scc}
We now explain how the spectral sequence of a simplicial object may be described in
more traditional terms, using chain complexes:

If $\C$ is a pointed category, let \w{\Ch(\C)} denote the category of
(non-negatively graded) chain complexes in $\C$: that is, commuting diagrams of the form
\mydiagram[\label{eqchaincx}]{
\dotsc A\sb{n}\ar[r]\sp{\df{n}} \ar[rd] |!{[d];[r]}\hole &
A\sp{n-1} \ar[r]\sp{\df{n-1}} \ar[rd] |!{[d];[r]}\hole &
A\sp{n-2}\dotsc &
\dotsc A\sp{2} \ar[r]^(0.65){\df{2}} \ar[rd] |!{[d];[r]}\hole &
A\sp{1}\ar[r]\sp{\df{1}} &  A\sp{0}~,\\
\dotsc \ast \ar[ru] & \ast \ar[ru] & \ast\dotsc & \dotsc \ast\ar[ru]  & \ast\ar[ru]  &
}
\noindent so \w{\df{i-1}\circ\df{i}=0} for \w[.]{i\geq 1}
For any \w{A\in\C} and \w[,]{n\geq 0} let \w{A\oS{n}} be the chain complex having
$A$ in dimension $n$, and $0$ elsewhere.

We denote by \w{\Chn{n}{k}(\C)} the category of \wwb{n,k}\emph{truncated}
chain complexes in $\C$ \wh that is, diagrams \wref{eqchaincx} starting at $n$ and
ending at $k$ \wh with truncation functor \w[.]{\tnk{n}{k}:\Ch(\C)\to\Chn{n}{k}(\C)}
\end{mysubsection}

\begin{remark}\label{radjoints}
If $\C$ has enough (co)limits, \w{\tnk{n}{k}} has a left adjoint
\w{\lnk{n}{k}:\Chn{n}{k}(\C)\to\Ch(\C)} defined by:
$$
\lnk{n}{k}(\Ds)\sb{i}~=~
\begin{cases}
D\sb{i}& \text{if}~k\leq i\leq n\\
\Coker(\partial\sb{k+1})& \text{if}~i=k-1\\
  0&\text{otherwise,}
\end{cases}
$$
\noindent as well as a right adjoint
\w{\rnk{n}{k}:\Chn{n}{k}(\C)\to\Ch(\C)} defined by:
$$
\rnk{n}{k}(\Cs)\sb{i}~=~
\begin{cases}
C\sb{i}& \text{if}~k\leq i<n\\
\Ker(\partial\sb{n}) & \text{if}~i=n+1\\
  0&\text{otherwise.}
\end{cases}
$$
\noindent Note that if $\C$ is a model category and \w{\Ds} is Reedy cofibrant in
\w{\Chn{n}{k}(\C)} (see \cite[\S 15.3]{HirM}), then \w{\Coker(\partial\sb{k+1})} is
the homotopy colimit of the truncated version of \wref[.]{eqchaincx} Similarly,
if \w{\Cs} is Reedy fibrant, \w{\Ker(\partial\sb{n})} is the homotopy limit.
For \w[,]{\Cs\in\Ch(\C)} we write \w{\cskn{n}{k}\Cs} for
\w[,]{\rnk{n}{k}\tnk{n}{k}\Cs} with the unit \w{\enk{n}{k}:\Cs\to\cskn{n}{k}\Cs}
a fibration when \w{\Cs} is fibrant.
\end{remark}

\begin{mysubsection}{Simplicial objects and chain complexes}
\label{ssocc}
If $\C$ is a pointed category with enough limits, the $n$-th \emph{Moore chains} object
of a restricted simplicial object \w{\Xd\in\C\sp{\rDop}} is defined to be:
\begin{myeq}\label{eqmoor}
C\sb{n}\Xd~:=~\cap\sb{i=1}\sp{n}\Ker\{d\sb{i}:X\sb{n}\to X\sb{n-1}\}~,
\end{myeq}
\noindent with differential
\w[.]{\df{n}:=d\sb{0}\rest{C\sb{n}\Xd}:C\sb{n}\Xd\to C\sb{n-1}\Xd}
The $n$-th \emph{Moore cycles} object is thus
\w[,]{Z\sb{n}\Xd~:=~\cap\sb{i=0}\sp{n}\Ker\{d\sb{i}:X\sb{n}\to X\sb{n-1}\}}
with \w{v\sb{n}:Z\sb{n}\Xd\to C\sb{n}\Xd} the inclusion. Note that \w{\df{n}}
factors through \w[.]{\wdf{n}:C\sb{n}\Xd\to Z\sb{n}\Xd}

The Moore chains functor \w{\Cs:s\sb{+}\C\to\Ch(\C)} has a left
adjoint (and right inverse) \w[,]{\E:\Ch(\C)\to s\sb{+}\C} with
\w[,]{(\E A\sb{\ast})\sb{n}=A\sb{n}} \w[,]{d\sb{0}\sp{n}=\df{n}} and
\w{d\sb{i}\sp{n}=0} for \w[.]{i\geq 1}

If $\C$ has enough colimits, the $n$-th \emph{latching object} for \w{\Xd\in\C\sp{\Dop}}
is the colimit
\begin{myeq}\label{eqlatch}
L\sb{n}\Xd~:=~\colimit{\theta\op:\bk\to\bn}\,X\sb{k}~,
\end{myeq}
\noindent where $\theta$ ranges over surjections \w{\bn\to\bk} in
$\Del$. Any iterated degeneracy \w{s\sb{I}:\bX\sb{k}\to\bX\sb{n}} factors through
the obvious map \w[.]{\sigma\sb{n}:L\sb{n}\Xd\to\bX\sb{n}}

If \w{\Xd} is an abelian group object in \w[,]{\C\sp{\Dop}} the
natural map \w{C\sb{n}\Xd\to\Coker(\sigma\sb{n})} is an
isomorphism, by \cite[Corollary (1.12)]{DolH}, so if we set
\w[,]{\oX{n}:=C\sb{n}\Xd} we have
\begin{myeq}\label{eqsplatch}
\bX\sb{n}~\cong~ L\sb{n}\Xd\oplus\oX{n}\hs \text{for each}\hs n\geq 0~,
\end{myeq}
\noindent and thus by induction (starting with \w[):]{\oX{0}=\bX\sb{0}}
\begin{myeq}\label{eqsplitlatch}
L\sb{n}\Xd~:=~
\coprod\sb{0\leq k\leq n-1}~~\coprod\sb{0\leq i\sb{1}<\dotsc<i\sb{n-k-1}\leq n-1}~
\oX{k}~,
\end{myeq}
\noindent with each summand on the right mapping to the left by
\w[.]{s\sb{i\sb{n-k-1}}\dotsc s\sb{i\sb{2}}s\sb{i\sb{1}}}

Note that the inclusion \w{\rDel\hra\Delta} induces a forgetful functor
\w[,]{\U:\C\sp{\Dop}\to\C\sp{\rDop}} and its left adjoint
\w{\LLc:\C\sp{\rDop}\to\C\sp{\Dop}} is given by
\w[.]{(\LLc\Xd)\sb{n}=X\sb{n}\amalg L\sb{n}\Xd}
The adjunction \w{\Cs\U:\C\sp{\Dop}\rightleftharpoons\Ch(\C):\LLc\E} can be viewed as
a version of the Dold-Kan correspondence in an arbitrary category \wh which is not
generally an equivalence, unless $\C$ is abelian.
See \cite[\S 1]{BJTurnHI} for further details.

All this makes sense also in any pointed \wwb{\infty,1}category, such as a
quasi-category $X$ (see Remark \ref{rinfc}). Moreover, if \w{\xd} in \w{sX}
and $\fy$ in $X$, are as in \S \ref{sssso}, and \w{\Xd:=\Map\sb{X}(\fy,\xd)} in
\w[,]{s\Sa} then \wref{eqsplatch} and \wref{eqsplitlatch} still hold up to weak
equivalence. Moreover, this holds even if $\fy$ is just a cogroup object in
\w[,]{\ho X} in some cases (see \cite{BJTurnHI}).
\end{mysubsection}

\begin{mysubsection}{A diagrammatic description of the differentials}
\label{sdiagdesc}
Given \w{\lra{\xd,\fy=\Sigma\hy}} in a quasi-category $X$ as in \S \ref{sssso},
let \w{\Xd:=\Map\sb{X}(\hy,\xd)} be the corresponding  homotopy coherent
simplicial space. By \cite{DKSmitH} (see also \cite[IX]{GJarS}), we can replace it
by a strict simplicial space in \w[,]{\Sa\sp{\Delta}} and further assume that \w{\Xd} is
Reedy fibrant.

In this case, as explained in \cite[\S 6]{BMeadS}, in order to obtain a
diagram \wref{diagnk} in $X$, it suffices to find a homotopy coherent diagram
\myvdiag[\label{diagnksp}]{
  \bA  \ar@/^1pc/[rr]\sp{d\sb{0}=0}\sb{\vdots}
  \ar@/_1pc/[rr]\sb{d\sb{n}=0} \ar[dd]\sb{\hat{f}}  && 0
  \ar@/^1pc/[rr]\sp{d\sb{0}=0}\sb{\vdots}  \ar@/_1pc/[rr]\sb{d\sb{n-1}=0} \ar[dd] &&
  0 \ar[dd] &\cdots\cdots & 0 \ar[dd] \\
  && && && \\
\bX\sb{n} \ar@/^1pc/[rr]\sp{d\sb{0}}\sb{\vdots}  \ar@/_1pc/[rr]\sb{d\sb{n}} && \bX\sb{n-1}
\ar@/^1pc/[rr]\sp{d\sb{0}}\sb{\vdots} \ar@/_1pc/[rr]\sb{d\sb{n-1}} && \bX\sb{n-2} &
\cdots \cdots & \bX\sb{n-r+1}~.
}
\noindent in \w[.]{\Sa} Since \w[,]{\fy=\Sigma\hy} 
so we may take \w{\bA=S\sp{p}} for \w[,]{p\geq 1} where $\hat{f}$ corresponds to
\w[.]{[f]\in\pi\sb{p}\Map\sb{X}(\hy,\xd)}

We extend the usual notation \w{\CsA{n}} for the cone on the $n$-fold suspension
by setting \w{\CsA{-1}:=\bA} and \w[,]{\CsA{-2}:=\ast} with \w{i\sp{n}:\sA{n}\to\CsA{n}}
the inclusion and \w{q\sp{n}:\CsA{n}\to\sA{n+1}} the quotient map.

It was shown in \cite[\S 2.B]{BJTurnHI} that one can extend
\w{\hat{f}:\bA\to\bX\sb{n}} to a diagram \wref{diagnksp}
as above if and only if we have a solid map of truncated chain complexes
\w{F:\bDs\to\tnk{n}{k}\Cs\Xd} in \w{\Chn{n}{k}(\Sa)} of the form
\mydiagram[\label{eqgrid}]{
\bD\sb{n}=\CsA{-1}=\bA \ar[dd]\sb{\df{n}} \ar[dr]\sp{q\sp{n-k-3}}
  \ar[rrr]\sp{F\sb{n}} & & &
  C\sb{n}\Xd \ar[dl]\sp{\wdf{n}} \ar[dd]\sp{\df{n}}\\
& \sA{0} \ar[dl]\sb{i\sp{0}} \ar[r]\sp{a\sb{n-1}}& Z\sb{0}{\Xd}\ar[dr]\sp{v\sb{0}}\\
\bD\sb{n-1}=\CsA{1} \ar@{.>}[d] \ar[rrr]\sp{F\sb{n-1}} & & & C\sb{n-1}\Xd\ar@{.>}[d]\\
  \bD\sb{k}=\CsA{n-k-1} \ar[dd]\sb{\df{k}} \ar[dr]\sp{q\sp{n-k-1}}
  \ar[rrr]\sp{F\sb{k}} & & & C\sb{k}{\Xd}\ar[dl]\sp{\wdf{k}}
    \ar[dd]\sp{\df{k}} \\
    & \sA{n-k} \ar[dl]\sb{i\sp{n-k}} \ar[r]\sp{a\sb{k-1}} &
    Z\sb{k-1}{\Xd} \ar[dr]\sp{v\sb{k-1}}\\
\bD\sb{k-1}=\CsA{n-k-2} \ar@{-->}[rrr]\sp{F\sb{k-1}}  & & & C\sb{k-1}\Xd
}
\noindent for \w[.]{k=n-r+1} Here \w[,]{j\sb{n}\circ F\sb{n}\simeq\hat{f}} for
\w{j\sb{n}:C\sb{n}\Xd\hra\bX\sb{n}} as above, and \w{\bDs} is a cofibrant replacement
for \w{\bA\oS{n}} in \w{\Chn{n}{k}(\Sa)} (in the notation of \S \ref{scc}),
which extends to \w{\Ch(\Sa)} in the obvious way. In addition, we have
\w{\lnk{n}{k}\bDs} ending in \w{\sA{n-k}} in dimension \w[,]{k-1} by Remark
\ref{radjoints}, with $F$ inducing the map \w[,]{c(F):=v\sb{k-1}\circ a\sp{k-1}} by
adjunction. This map represents \w{\fd\sp{r}([f])} in
\w[,]{E\sp{1}\sb{n-r,p+r-1}} and it must be nullhomotopic in order for \w{F\sp{k-1}}
to exist.
\end{mysubsection}

\begin{remark}\label{rerequivs}
Since \w{\bDs}  is cofibrant, \w{\Cs\Xd}  is fibrant, and \w{i\sp{n-k}} is a
cofibration, we have a (homotopy) pullback diagram
\mydiagram[\label{eqhpb}]{
  \map\sb{\Chn{n}{k-1}(\C)}(\tnk{n}{k-1}\bDs,\tnk{n}{k-1}\Cs\Xd) \ar[d]
  \ar@{->>}[rr]\sp{(\tnk{n}{k})\sb{\ast}} &&
\map\sb{\Chn{n}{k}(\C)}(\tnk{n}{k}\bDs,\tnk{n}{k}\Cs\Xd) \ar[d]\sb{c} \\
\ast\simeq\map\sb{\C}(\bD\sb{k-1},C\sb{k-1}\Xd) \ar@{->>}[rr]\sp{(i\sp{n-k})\sp{\ast}}
&&\map\sb{\C}(\sA{n-k},C\sb{k-1}\Xd)~,
}
\noindent for \w{c(F):=v\sb{k-1}\circ a\sp{k-1}} as above. In fact, by Remark
\ref{radjoints} we have a natural identification
\begin{myeq}\label{eqadjcha}
\map\sb{\Chn{n}{k}(\C)}(\tnk{n}{k}\bDs,\tnk{n}{k}\Cs\Xd)~=~
\map\sb{\Ch(\C)}(\lnk{n}{k}\bDs,\Cs\Xd)~,
\end{myeq}
\noindent and the map $c$ on the right hand side of \wref{eqadjcha} is induced by
restriction to dimension \w[.]{k-1}

Note that \w{\map\sb{\Chn{n}{k}(\C)}(\tnk{n}{k}\bDs,\tnk{n}{k}\Cs\Xd)} splits as
a disjoint union of subspaces corresponding to distinct \w{[f]\in E\sp{1}\sb{n,p}}
which survive to \w{E\sp{r}} (each of which may consist of several connected
components).
Evidently, a map \w{g:\xd\to\yd} in \w{sX} or \w{\spX} (or the corresponding map
\w{\hat{g}:\Xd\to\Yd} in \w{\Sa\sp{\Dop}} or \w[)]{\Sa\sp{\rDop}} which induces
a weak equivalence on the right vertical arrow of \wref{eqhpb} for all \w{n\geq r}
will induce a weak equivalence
$$
\map\sb{\Chn{n}{k-1}(\C)}(\tnk{n}{k-1}\bDs,\tnk{n}{k-1}\Cs\Xd)~\xra{g\sb{\ast}}~
\map\sb{\Chn{n}{k-1}(\C)}(\tnk{n}{k-1}\bDs,\tnk{n}{k-1}\Cs\Yd)~,
$$
\noindent and thus an isomorphism in the \ww{E\sp{r+1}}-terms of the spectral sequences
of \w{\lra{\xd,\fy}} and \w[.]{\lra{\yd,\fy}}  However, this is far from
being necessary.
\end{remark}

%
%
\sect{The spectral sequence of a simplicial object and localization}
\label{cssl}

Let \w{\lra{X,\Fib,\W}} be a complete quasi-category with fibrations and weak
equivalences  as in Definition \ref{def1.2}, satisfying Assumptions \ref{ass3.1}.
By Corollary \ref{cor1.6}, \w{\spX} (and its various truncations)
may be equipped with Reedy fibrations and levelwise equivalences to
form a quasi-category with fibrations and weak equivalences.

It turns out that there are two types of localization relevant to the spectral
sequence of a simplicial object \wh both using the fact that the differentials
are determined by mapping out of finite diagrams as in \wref{diagnksp}
or \wref[.]{eqgrid}

\begin{mysubsection}{Postnikov sections and spectral sequences}
\label{spsss}
The first type of localization is based on the oldest known form of localization
in homotopy theory \wh the Postnikov section:

As noted in \S \ref{sssso},  we may assume that \w{\fy=\Sigma\hy}
is a suspension in $X$, and so  \w{\bA=S\sp{p}} in \wref{eqgrid}
for \w[,]{p\geq 0} so we may replace the map of truncated chain complexes
\w{F:\bDs\to\tnk{n}{k}\Cs\Xd} in \w{\Chn{n}{k}(\Sa)} in \wref{eqgrid}
under the $p$-fold loop-suspension adjunction in \w{\Sa} by
\w[,]{\widetilde{F}:\tDs\to\tnk{n}{k}\Omega\sp{p}\Cs\Xd} where \w[,]{\tDs}
a cofibrant replacement of \w{\tA\oS{n}} for\w[,]{\tA=\bS{0}} is the $p$-fold
desuspension of \w[.]{\bDs}  Note that all objects in \w{\tDs} are of dimension
\w[,]{\leq n-k=r-1} so $\widetilde{F}$ factors through the \wwb{r-1}Postnikov section
of \w[.]{\tnk{n}{k}\Omega\sp{p}\Cs\Xd} The same is true of the corresponding map
$\widehat{F}$ in the right hand side of \wref[,]{eqadjcha} and thus also for
\w[,]{c(\widetilde{F})} adjoint to \w[.]{c(F)}

Note that the adjunction
\begin{myeq}\label{eqchsimp}
\E\colon\Ch(\Sa)~\leftrightarrows~\Sa\sp{\rDop}\colon\Cs
\end{myeq}
\noindent of \S \ref{ssocc} is homotopy meaningful, by \cite[Lemma 2.7]{StoV},
and allows us to convert $\widetilde{F}$  into a map in \w{\Sa\sp{\rDell{k,n}}}
of the form \wref{diagnksp} in which we have a cofibrant replacement for the top
row (since $\E$ preserves Reedy cofibrancy). Using Remark \ref{rerterm}, we deduce:
\end{mysubsection}

\begin{prop}\label{rtrunc}
Given \w{\xd\in sX} as in \S \ref{sssso} and $\hy$ as above, the representation of
\w{\fd\sp{r}([f])} in \w{E\sp{1}\sb{n-r,p+r-1}} associated to the
extension \wh and thus \w{d\sp{r}(\lra{f})\in E\sp{r}\sb{n-r,p+r-1}} itself, as we
run over all possible extensions \wh depends only on
\w[,]{\Po{r-1}\Omega\sp{p}\Map\sb{X}(\hy,\xd)}
which thus determines \w[,]{E\sp{r+1}\sb{\ast,\ast}} and in particular
allows us to determines whether \w{[f]} survives to \w[.]{E\sp{r+1}}
\end{prop}

This result can be refined using the following:

\begin{defn}\label{dzmnp}
Given $\hy$ as above, for any \w{p\geq 0} and
\w{1\leq m \le n} we let \w{H(n,m,\Sigma\sp{p}\hy)} denote the diagram
\mywdiag[\label{eq1}]{
&  \Sigma\sp{p}\hy \ar@<3ex>[rr]\sp{d\sb{0} = 0} \ar@<0.5ex>[rr]\sp{d\sb{1} = 0}
  \ar@<-2.5ex>[rr]\sb{d\sb{n} = 0}\sp{\vdots} && 0 \ar@<2ex>[rr]\sp{d\sb{0}}
  \ar@<-1ex>[rr]\sb{d\sb{n}}\sp{\vdots} && 0
  \ar@<2ex>[rr]\sp{d\sb{0}} \ar@<-1ex>[rr]\sb{d\sb{n-2}}\sp{\vdots}&& \cdots
  \ar@<2ex>[rr]\sp{d\sb{0}} \ar@<-1ex>[rr]\sb{d\sb{n-m+1}}\sp{\vdots} && 0~\\
   \text{dimension:} & n && n-1 && n-2 &&\dotsc&& n-m
}
\noindent in \w{\spnX{n-m,n}} (unique up to a contractible space of choices,
by the universal property of $0$), and let
\begin{myeq}\label{eqhr}
\begin{split}
\HH\sp{r}(\fy)~:=~&
\bigcup\sb{p\geq 1}~\bigcup\sb{n\geq 0}~\left[\bigcup\sb{1\leq m\leq\min\{p,r\}}~
  \{H(n+m,m-1,\Sigma\sp{p-m+1}\hy)\}\right]\\
&~\cup~
\bigcup\sb{n\geq r-1}\,\{H(n,r-1,\Sigma\sp{p}\hy)\}
~\cup~\bigcup\sb{n<r-1}\,\{H(n,n-1,\Sigma\sp{p}\hy)\}
\end{split}
\end{myeq}
\noindent with \w{\HH(\fy):=\bigcup\sb{r=2}\sp{\infty} \HH\sp{r}} the collection
of all such diagrams.
\end{defn}

\begin{remark}\label{rdoublec}
The reader will note that the list in \wref{eqhr} has repetitions; the reason is that
the first set of objects of the form \w{H(n+m,-,-)} are used
to identify when \w{[f]\in E\sp{1}\sb{n,p}} is in the image of the (earlier)
differentials, while the next two sets, of the form \w{H(n,-,-)} are used to verify
that \w{[f]} is a \ww{d\sp{r}}-cycle. Thus for the first set, we are only
interested in maps in the top right corner of \wref{eqhpb} with non-trivial image
under $c$, while for the second case we want the fiber of $c$ (see Remark
\ref{rerequivs}).
One could use this distinction to further refine the localizations defined below, but
we shall not pursue this idea further here.
\end{remark}

\begin{defn}\label{drstem}
The various inclusions
\begin{myeq}\label{eqrestdiag}
\rDell{m,n}~\hookrightarrow~\rDell{m',n}
\end{myeq}
\noindent induce a partial order on the subset of diagrams in \w{\HH(\fy)}
with a fixed $p$ and $n$.

For any quasi-category $X$ and \w{a,b\in X} the adjunction
\w{\Map\sb{X}(\Sigma a,b)\simeq\Omega\Map\sb{X}(a,b)} induces natural maps
\begin{myeq}\label{eqstems}
\Po{r}\Map\sb{X}(\Sigma a,b)\xra{\simeq}\Po{r}\Omega\Map\sb{X}(a,b)\to
  \Po{r-1}\Omega\Map\sb{X}(a,b)\xra{\simeq}\Omega\Po{r}\Map\sb{X}(a,b)
\end{myeq}
\noindent for each \w[.]{r\geq 1}

Thus given \w{\fy=\Sigma\hy\in X} and \w{\xd\in\spX} as in \S \ref{sssso},
for each \w{r\geq 0} we define the \emph{$r$-stem for} \w{\lra{\xd,\fy}} to be the
system consisting of
\begin{myeq}\label{eqrstems}
\Po{m}\Map\sb{\LW(\spnX{n-m+1,n})}
(H(n,m,\Sigma\sp{p}\hy),\tnk{\ast}{n-m+1,n}\xd)
\end{myeq}
\noindent for all \w[,]{H(n,m,\Sigma\sp{p}\fy)\in\HH\sp{r}(\fy)}
under the various maps induced by \wref{eqrestdiag} and \wref[.]{eqstems}
This is a more precise version of the ``spiral $r$-system'' of
\cite[Section 4]{BBlanS}.
\end{defn}

We then deduce from Proposition \ref{rtrunc} and \cite[Theorem 6.8]{BMeadS}
the following refinement of \cite[Theorem 4.13]{BBlanS}:

\begin{thm}\label{tstem}
Given \w{\fy\in X} and \w{\xd\in sX} as in \S \ref{sssso}, for each \w[,]{r\geq 2}
the \ww{E\sp{r}}-term of the associated spectral sequence is determined
by the \wwb{r-2}stem of \w[.]{\lra{\xd,\fy}}
\end{thm}

\begin{mysubsection}{The Postnikov localization}
\label{sploc}
We can reformulate Theorem \ref{tstem} in terms of a pair of Bousfield localizations,
as follows:

The Postnikov section functor \w{\Po{r}:\Sa\to\Sa} is a
nullification with respect to \w[,]{\bS{r+1}} so it is
cocontinuous by \cite[Proposition 3.4.4]{HirM}, and continuous by
\cite[Theorem 9.9]{BousTH}. Thus we may think of it as a left
Bousfield localization \w{\LLc\sp{r}} on the quasi-category $Y$ of
pointed $\infty$-groupoids, with the usual class of fibrations
\w{\Fib} and weak equivalences $\W$ (corresponding to those of the
usual model category structure on \w[),]{\Sa} extended objectwise
to each functor category \w[.]{\Chn{n}{k}(Y)}

Similarly, if for each \w[,]{p\geq 0} we denote the $p$-\emph{connected cover} functor
by \w[,]{(-)\lra{p}:\Sa\to\Sa} we may replace
\w{\Po{r-1}\Omega\sp{p}\Map\sb{X}(\widehat{\fy},\xd)} by
\w{\Po{r-1}(\Map\sb{X}(\widehat{\fy},\xd)\lra{p})} in Proposition \ref{rtrunc}.

However, \w{(-)\lra{p}} is just the colocalization, or cellularization, with respect to
\w{\bS{p+1}} (see \cite[\S 3.1.7]{HirM}), so we may think of it as a right Bousfield
localization \w{\R\sb{p}} on \w{\lra{Y,\Fib,\W}}  as above, again extended to each
\w[.]{\Chn{n}{k}(Y)}

Now fix \w[,]{r\geq 2} and consider the quasi-category
\begin{myeq}\label{eqproduct}
Z\sp{r}~:=~\prod\sb{H(n,m,\Sigma\sp{p}\hy)\in\HH\sp{r}(\fy)}~\Chn{n}{n-m+1}(Y)\bp{p}
\end{myeq}
\noindent (the weak product of presheaf categories).

We may apply to \w{Z\sp{r}} the combined left and right Bousfield
localization taking the form \w{\LLc\sp{m}\circ\R\sb{p}} on the factor
\w[.]{\Chn{n}{n-m+1}(Y)\bp{p}} This defines the $r$-th \emph{Postnikov localization}
functor \w[.]{\PPc\sp{r}:Z\sp{r}\to Z\sp{r}} By Theorem \ref{thm3.5}
(and the corresponding
straightforward analogue for right localizations), \w{Z\sp{r}} has the structure of
a quasi-category with fibrations, in which the class of weak equivalences
\w[,]{\W\sp{r}} called \ww{\PPc\sp{r}}-\emph{equivalences}, are
Reedy (i.e., degree-wise) weak equivalences of the truncated chain complexes
\w{\Po{m-1}\Cs\lra{p}} on the factor \w[.]{\Cs\in\Chn{n}{n-m+1}(Y)\bp{p}}

Note that for any quasi-category $X$ as in \S \ref{sssso}, we have a sequence of
functors
\begin{equation*}
\begin{split}
\spX~&\xra{\Map\sb{X}(\fy,-)}~s\sb{+}\Sa~\xra{\Cs}~
\Ch(Y)~\xra{\tnk{n}{n-m+1}}~\Chn{n}{n-m+1}(Y)\\
&~\xra{\Po{m-1}}~\Chn{n}{n-m+1}(Y)~\xra{(-)\lra{p}}~\Chn{n}{n-m+1}(Y)~
\end{split}
\end{equation*}
\noindent for $m$, $n$, and $p$ as in \wref[,]{eqproduct} which together define
a functor \w[.]{G\sp{r}:\spX\to Z\sp{r}} We think of the component of
\w{G\sp{r}(\xd)} in \w{\Chn{n}{n-m+1}(Y)\bp{p}} as providing the
\wwb{m,p,n}\emph{window} for \w{\xd} (in the sense of \cite[\S 2.2]{BBlanS}).

We see by Proposition \ref{rtrunc} that the spectral sequence for \w{\xd\in\spX}
(with respect to a fixed \w{\fy\in X} as in \S \ref{spsss}) is determined through
the \ww{E\sp{r+2}}-page by \w[,]{G\sp{r}\xd} and conclude from Theorem \ref{tstem}:

\begin{corollary}\label{cstempr}
The \ww{\PPc\sp{r}}-equivalences induce isomorphisms of
the associated spectral sequences from the \ww{E\sp{r+2}}-term on.
\end{corollary}
\end{mysubsection}

\begin{remark}\label{rtruncations}
We might try to use \w{G\sp{r}:\spX\to Z\sp{r}} to lift the notion of a
\ww{\PPc\sp{r}}-equivalence to \w{\spX} itself, or just to \w[.]{\Ch(\Sa)}
However, because the a weak equivalence at all \wwb{m,p,n}-windows is just
a Reedy equivalence of (restricted) simplicial spaces, we would not gain
anything from the corresponding localization of \w[.]{\spX}

On the other hand, the discussion in \S \ref{ssocc} allows us to reformulate
the Postnikov localization \w{\PPc\sp{r}} in terms of simplicial truncations:
more precisely, if we set
\begin{myeq}\label{eqsproduct}
  \widehat{Z}\sp{r}~:=~
  \prod\sb{H(n,m,\Sigma\sp{p}\hy)\in\HH\sp{r}(\fy)}~\spn{n-m+1,n}(Y)\bp{p}~,
\end{myeq}
\noindent the restrictions \w{\Cs:\spn{n,k}(Y)\to\Chn{n}{k}(Y)} of the Moore chains
to each factor combine to define a functor \w[.]{F:\widehat{Z}\sp{r}\to Z\sp{r}}

Because each restricted \w{\Cs} is right adjoint (and left inverse) to
\w[,]{\tnk{n}{k}\circ\E\circ\rnk{n}{k}} the functor $F$ satisfies assumptions (a)-(c)
of Proposition \ref{pinduce}, so we can use it to lift the \ww{\PPc\sp{r}}-structure of
a quasi-category with fibrations from \w{Z\sp{r}} to \w[.]{\widehat{Z}\sp{r}}
Of course, we could also have constructed it directly as in \S \ref{sploc}.
\end{remark}

\begin{mysubsection}{The $\E\sp{r}$-localization}
\label{serl}
The second form of localization we need is the following:

Note that since $X$ is locally presentable, for each \w{m \le n} we have a
left Kan extension adjunction
$$
\LKE\sb{n, m} :\spn{n,m}(X) \leftrightarrows\spX : i\sp{\ast}\sb{n, m}
$$
\noindent by \cite[Proposition 6.4.9]{CisiH}. By Theorem \ref{thm4.10}, we therefore
have a right Bousfield localization of \w{\spX} at the family
$$
\E\sp{r}~:=~
\{\LKE\sb{m, n}(H(n,m,\Sigma\sp{p}\hy)\}\sb{H(n,m,\Sigma\sp{p}\hy)\in\HH\sp{r}(\fy)}~.
$$
(see \S \ref{dzmnp}). Thus \w{\spX} has a new structure of a quasi-category with
fibrations and weak equivalences, in which the latter are the
\ww{\E\sp{r}}-\emph{equivalences}.
The left Kan extension along a fully faithful inclusion of quasi-categories is also
fully faithful \cite[Proposition 4.3.2.17]{LurieHTT}, so we may deduce
from \S \ref{sssso} and Remark \ref{rerterm}:
\end{mysubsection}

\begin{corollary}\label{cerloc}
The \ww{\E\sp{r}}-equivalences induce \ww{E\sp{r}}-isomorphisms of
the associated spectral sequences.
\end{corollary}

%
%
\sect{The Spectral Sequence of a Cosimplicial Object}
\label{cssrfcs}

We now investigate the dual to the spectral sequences considered so far:
namely, the homotopy spectral sequence of a cosimplicial object in an
\wwb{\infty,1}category.  As in \cite{BMeadS}, we require a description which allows
us to analyze the differentials in the spectral sequence, when applied to a
representative in the \ww{E\sb{1}}-term \wh much as we did for in the simplicial
case in Sections \ref{csssso} and \ref{cssl}.

This was discussed briefly in \cite[Section 9]{BMeadS}, but the treatment there is
not sufficient for our purposes here, which depend on three basic requirements:
\begin{enumerate}
\renewcommand{\labelenumi}{(\arabic{enumi})~}
\item We want the differentials \wh and thus the spectral sequence as a
  whole \wh to depend only on the underlying \emph{restricted} cosimplicial object
  (forgetting the codegeneracies).
\item It should be possible to recover the $r$-th differential \wb{r\geq 2} from
the \wwb{r-2}Postnikov truncation in the associated simplicial category $\mX$.
\item We want a model-independent, and in particular homotopy invariant,
  description of the differentials.
\end{enumerate}
With these goals in mind, we now  give a more detailed construction from scratch:

\begin{mysubsection}{Cosimplicial objects in quasi-categories}
\label{scoqc}
Suppose that we have a pointed, locally presentable quasi-category $X$, and a
compact, connected (abelian) homotopy cogroup object \w{\fy=\Sigma\hy}
as in \S \ref{sssso}.
Given a cosimplicial object \w{\xu\in cX} (see \S \ref{snac}), we obtain
a (homotopy coherent) pointed cosimplicial space \w{\wu:=\MAP\sb{X}(\hy,\xu)}
in \w[,]{c\Sa:=\Sa\sp{B(\Del)}} using Lemma \ref{lem8.2}. By \cite[Theorem 6.7]{Riehl1},
we can associate to \w{\wu} a homotopy coherent diagram in Kan complexes.
Using homotopy coherence theory (see \cite[Section 9]{GJarS}), we can replace this in
turn with an equivalent strict cosimplicial space \w[.]{\Wu} We then define the
spectral sequence associated to \w{\lra{\xu,\fy}} to be the Bousfield-Kan homotopy
spectral sequence of \w{\Wu} (more precisely: a Reedy fibrant replacement thereof),
with
\begin{myeq}\label{eqbkss}
E\sb{2}\sp{n,n+p}~=~\pi\sp{n}\pi\sb{n+p}\Wu~\cong~
~\pi\sp{n}\pi\sb{n+p}\MAP\sb{X}(\hy,\xu)~\cong~
\pi\sp{n}[\Sigma\sp{n+p}\hy,\xu]
\end{myeq}
\noindent (the indexing has been chosen because this term contributes to
\w[),]{\pi\sb{p}\Tot\Wu} with
\w[.]{d\sb{r}:E\sb{r}\sp{n,n+p}\to E\sb{r}\sp{n+r,n+p+r-1}}
See \cite[IX]{BKanH} for further details.

We again take the point of view, explained in \S \ref{sssso}, that the differential
\w{d\sb{r}(\gamma)} for \w{\gamma\in E\sb{r}\sp{n,n+p}} is to be described in
terms of all possible values \w{\fd\sb{r}[f]\in E\sb{1}\sp{n+r,n+p+r-1}} for
the various representatives \w{[f]\in E\sb{1}\sp{n,n+p}} of $\gamma$.

In light of the above, for the remainder of this section we fix a Reedy fibrant
and cofibrant cosimplicial space \w[.]{\Wu\in\Sa\sp{\Delta}} Because
\w{\pi\sb{k}\bW\sb{n}\cong[\Sigma\sp{k}\hy,x\sp{n}]} and $\fy$ is an abelian cogroup
object in $X$, \w{\pi\sb{k}\bW\sb{n}} is an abelian group for each \w{n\geq 0} and
\w[.]{k\geq 1}
\end{mysubsection}

It is possible to develop a full description of the spectral sequence of \w{\Wu}
(and indeed of \w[)]{\lra{\xu,\fy}} to the category of cochain complexes in \w[,]{\Sa}
as we did in the simplicial case in \S \ref{scc} (see \cite{BBSenH}).
However, in the interests of brevity we describe only the cosimplicial
version of \S \ref{ssocc}:

\begin{mysubsection}{Cosimplicial objects and cochain complexes}
\label{scoch}
If $\C$ is a pointed category, let \w{\cCh(\C):=\Ch(\C\op)} denote the category of
(non-negatively graded) cochain complexes in $\C$.

The \emph{$n$-th normalized cochain object} of a cosimplicial object \w{\Wu} in \w{c\C}
is defined
\begin{myeq}\label{eqnormcoch}
N\sp{n}(\Wu)~:=~\bigcap\sb{i=0}\sp{n-1}\Ker(s\sp{i}:\bW\sp{n}\to\bW\sp{n-1})~,
\end{myeq}
\noindent and we have
\begin{myeq}\label{eqceone}
E\sb{1}\sp{n,p}~\cong~\pi\sb{p}N\sp{n}(\Wu)
\end{myeq}
\noindent in our spectral sequence (see \cite[X, 6.3(i)]{BKanH}).

Alternatively, if we denote by \w{D\sp{n}(\Wu)} the (homotopy) image of
$$
\coprod\sb{i=1}\sp{n-1}\,\bW\sp{n}~\xra{\bot\sb{i}\,d\sp{i}}~\bW\sp{n}~,
$$
the \emph{$n$-th Moore cochain object} of \w{\Wu}  is defined to be the (homotopy)
cofiber
\begin{myeq}\label{eqmoorecc}
C\sp{n}(\Wu)~:=~\Coker(D\sp{n}(\Wu)\to\bW\sp{n})~,
\end{myeq}
\noindent with differential \w{\delta\sp{n}:C\sp{n}\Wu\to C\sp{n+1}(\Wu}
induced by \w[.]{d\sp{0}}

Note that the Moore cochain functor \w{\Cus:c\sp{+}\C\to\cCh{\C}} has a right
adjoint (and left inverse) \w[.]{\E:\cCh{\C}\to c\sp{+}\C}
Likewise, the forgetful functor \w{\U:c\C\to c\sp{+}\C} (induced by
\w[)]{\rDel\hra\Del} has a right adjoint \w{\F:c\sp{+}\C\to c\C} adding codegeneracies
(see \cite[\S 1.8]{BSenH}).
\end{mysubsection}

\begin{lemma}\label{lem7.1}
If a cosimplicial space \w{\Wu\in c\Sa} consisting of Eilenberg-MacLane spaces of type $n$
in each simplicial degree, and all coface and codegeneracy maps are homomorphisms,
then \w[.]{\bW\sp{n}\simeq N\sp{n}(\Wu) \times D\sp{n}(\Wu)}
\end{lemma}

\begin{proof}
The dual of  \cite[Theorem III.2.1]{GJarS} yields an isomorphism
\w[,]{N\sp{n}(\xu)\simeq\bW\sp{n}/D\sp{n}(\Wu)} and thus a splitting
of  \w[.]{\bW\sp{n}\to\bW\sp{n}/D\sp{n}(\Wu)}
\end{proof}

\begin{prop}\label{lem7.2}
For \w{\Wu} as in \S \ref{scoqc}, we have a homotopy equivalence
\w[.]{\bW\sp{n}\simeq N\sp{n}(\Wu) \times D\sp{n}(\Wu)}
\end{prop}

\begin{proof}
We let \w{\Po{m}:\Sa\to\Sa} denote the $m$-th Postnikov section functor.
We prove by induction on \w{m\geq 0} that the statement holds for \w[:]{\Po{m}(W\sp{n})}

The case \w{m = 0} follows from Lemma \ref{lem7.1} and the assumption that
\w{\pi\sb{0}\bW\sp{n}} is an abelian group. In step $m$, we have a fibre
sequence \w[,]{\Fu\sb{m}\to\Po{m}\Wu\to\Po{m-1}\Wu} where \w{\Fu\sb{m}} is an
Eilenberg-Mac~Lane space in each cosimplicial degree.

Since (homotopy) colimits commute with homotopy fibre sequences in \w[,]{\Sa} the functor
\w{D\sp{n}} commutes with fibre sequences, and we have a comparison of fibre sequences:
\begin{equation*}
\xymatrix@R=15pt@C=20pt{
D\sp{n}\Fu\sb{m} \times N\sp{n}\Fu\sb{m}  \ar[d]  \ar[r] &
D\sp{n}\Po{m}\Wu \times N\sp{n}\Po{m}\Wu \ar[d] \ar[r] &
D\sp{n}\Po{m-1}\Wu \times N\sp{n}(\Po{m-1}(\Wu)) \ar[d] \\
F\sb{m}\sp{n} \ar[r] & \Po{m}W\sp{n} \ar[r] & \Po{m-1}W\sp{n}
}
\end{equation*}
The left and right vertical maps are homotopy equivalences, by Lemma \ref{lem7.1}
and by the induction hypothesis, respectively, so the middle one is as well.

Since filtered homotopy limits and finite homotopy (co)limits of spaces commute,
\w{D\sp{n}} and \w{N\sp{n}} commute with filtered homotopy limits. Thus
\begin{equation*}
\begin{split}
\holim\ \Po{n}\bW\sp{n}~&\simeq~\holim\,[ N\sp{n}(\Po{n}(\Wu))\times D\sp{n}(\Po{n}(\Wu))]\\
&\simeq~N\sp{n}(\holim \Po{n}\Wu)\times D\sp{n}(\holim\Po{n}\Wu)
~=~N\sp{n}(\Wu)\times D\sp{n}(\Wu)~,
\end{split}
\end{equation*}
\noindent which completes the proof.
\end{proof}

From \wref{eqmoorecc} we deduce the following generalization of the dual of
\cite[Corollary (1.12)]{DolH}:

\begin{corollary}\label{cmoorenorm}
The natural map \w{\Cus(\Wu)\to N\sp{\ast}(\Wu)} is a levelwise weak equivalence.
\end{corollary}

\begin{mysubsection}{The \w{\Tot} tower}
\label{stott}
Let \w{\Du\in\Set\sp{\Delta}} denote the cosimplicial space having
\w{\Del\sp{n}} in degree $n$, and recall that \w{\Tot(\Wu):=\map\sb{c\Ss}(\Du,\Wu)}
for \w{\Du} as in \S \ref{snac}, where \w{\map\sb{c\Ss}} is the simplicial enrichment
for the Reedy model structure on \w[.]{c\Ss}
Similarly, \w[.]{\Tot\sp{n}(\Wu):=\map\sb{c\Ss}(\sk{n}\Du,\Wu)}

Thus a $k$-simplex of \w{\Tot\sp{n}(\Wu)} is a choice of maps
\begin{equation}\label{eqone}
f\sb{m} :\sk{n}\Del\sp{m} \times \Deln{k} \to \bW\sp{m}
\end{equation}
for each \w{m\geq 0} such that
$$
f\sp{j}\circ(\sk{n}\phi \times\id) = \bW(\phi) \circ
f\sp{m} :\sk{n}\Del\sp{m} \times \Deln{k} \to \bW\sp{j}~.
$$
We use the notations \w{\Del\sp{m}} and \w{\Deln{m}} for the $m$-simplex thought
of as a space and as a combinatorial book-keeping device, respectively.

Since \w[,]{\sk{n}(\Del\sp{N})=\colim\sb{k\leq n}\Del\sp{k}}
the map of simplicial sets is completely determined by the maps \w{f\sb{k}}
for \w[.]{k\leq n}

A representative $f$ of a homotopy class \w{[f]\in\pi\sb{k}\Tot\sp{n}\Wu} is determined
by a collection of maps as above, whose restriction to
\w{\sk{n}\Del\sp{m}\times \partial \Deln{k}} is $0$.
Similarly, a homotopy \w{F:f\sim f'} between two representatives of
\w{[f]\in\pi\sb{k}\Tot\sp{n}\Wu} is determined by a collection of compatible maps
$$
F\sb{m} :\sk{n}\Del\sp{m}\times\Deln{k}\times[0,1]\to\bW\sp{m}
$$
for \w[.]{m \le n} The homotopy groups of \w{\Tot\sp{n}\Wu} are thus determined
by the truncation \w{\tnk{h}{\le n}\Wu} of \w{\Wu} in cosimplicial dimensions
$\leq n$.

We write \w{\Wu\bp{n}} for the Reedy fibrant replacement of the left Kan extension of
\w{\tnk{h}{\le n}(\Wu)} to an object of \w[.]{\Ss\sp{\Delta}} By the previous remark,
\w{\pi\sb{\ast}\Tot\sp{n}\Wu} depends only on this truncation: that is, the natural map
\w{\Tot\sp{n}\Wu\bp{n}\to\Tot\sp{n}\Wu} is an equivalence.
Thus, we can identify the homotopy spectral sequence of the cosimplicial space \w{\Wu}
with the spectral sequence of the tower of fibrations
\begin{myeq}\label{tottower}
  \cdots\to\Tot\sp{n}(\Wu\bp{n})\to\Tot\sp{n-1}(\Wu\bp{n-1})\to\cdots\to
  \Tot\sp{0}(\Wu\bp{0}~.
\end{myeq}

Now let \w[.]{\oW{n}:=N\sp{n}\Wu\bp{n}} We extend the usual notation
\w{P\Omega\sp{n}} (for the composite of the path space functor
with $n$-fold loop space, when \w[)]{n\geq 0} by letting \w{P\Omega\sp{-1}X:=X} and
\w{P\Omega\sp{n}:=\ast} for \w[.]{n <-1} We then set
$$
M\sp{r}\bp{n}\oW{n}:=\prod\sb{0 \le k \le r}
\prod\sb{0 \le i\sb{1} < i\sb{1} \cdots < i\sb{k} \le r} P\Omega\sp{n+k-r-1}\oW{n}~,
$$
where the codegeneracy map \w{s\sp{t}: M\sp{r}\bp{n}\oW{n}\to M\sp{r}\bp{n-1}\oW{n}} is
given on the factor corresponding to \w{I=(i\sb{1}, \cdots, i\sb{k})}
by projection onto the unique factor \w{J=(j\sb{1}, \cdots j\sb{k+1})}
such that \w[.]{s\sp{I}\circ s\sp{t}=s\sp{J}}

Thus \w{\Mu\bp{n}\oW{n}} is obtained by applying the functor \w{\F:c\sp{+}\Sa\to c\Sa}
of \S \ref{scoch} to the restricted cosimplicial object
\w[,]{\Qu\in c\sb{+}\Sa} where \w[.]{\bQ\sp{m} := P\Omega\sp{n-m-1}\oW{n}}
In particular, the map into \w{M\sp{i}\bp{n}\oW{n}} for \w{i>n} is
determined by the cosimplicial identities, applying an appropriate
iterated codegeneracy to \w{M\sp{i}\bp{n}\oW{n}} until we land in \w[.]{\oW{n}}

We claim that the fibre sequence
$$
N\sp{n}(\Wu\bp{n})\to\Tot\sp{n}(\Wu\bp{n})\to\Tot\sp{n-1}(\Wu\bp{n-1})
$$
can be identified with the fibre sequence
$$
N\sp{n}(\Wu\bp{n}) \to\Tot\sp{n}(\Wu\bp{n-1} \times \Mu\bp{n}\oW{n})
\to\Tot\sp{n-1}(\Wu\bp{n-1})~.
$$
\noindent In fact, we have a natural map
\w[,]{\Wu\bp{n-1} \times\Mu\bp{n}\oW{n}\to\Wu\bp{n}}
which is an equivalence in cosimplicial dimensions \w[:]{m \leq n}

The case \w{m = n} is Proposition \ref{lem7.2},  and for \w[,]{m < n} this holds because
\w{\bW\bp{n-1}\sp{m} = \bW\bp{n}\sp{m}} and \w{M\sp{m}\bp{n}\oW{n}}
is weakly contractible. The result now follows from the description of the
homotopy groups of \w{\Tot\sp{n}} in \S \ref{stott}.

By the preceding paragraph, we can thus assume by induction that
\begin{myeq}\label{eq2}
\Wu\bp{n} \simeq \Wu\sb{[n-1]} \times\Mu\sb{n}\oW{n}
\end{myeq}
\noindent for all \w[.]{n\geq 1}
\end{mysubsection}

We have the following analogue of \S \ref{sdiagdesc}:

\begin{mysubsection}{A diagrammatic description of the differentials}
\label{sdiagdes}
Starting with \w[,]{F\sp{n-1}=d\sp{0}:C\sp{n-1}\bW\bp{n-1}\sp{n-1}
  \to P\Omega\sp{-1}\bW\bp{n-1}\sp{n-1}}
we may define by downward induction a map of complexes
\w{F:\Cus\Wu\bp{n-1}\to\Dus} given by:
\mysdiag[\label{complexes}]{
  C\sp{k+1}\Wu\bp{n-1}\ar[rrr]\sp{F\sp{k+1}} &&& P\Omega\sp{n-k-3}\oW{n} & =
  D\sp{k+1} \\
&&& \Omega\sp{n-k-2}\oW{n} \ar@{^{(}->}[u] & \\
C\sp{k}\Wu\bp{n-1} \ar[uu]\sp{\delta\sp{k}} \ar[rrr]\sb{F\sp{k}} &&&
P\Omega\sp{n-k-2}\oW{n} \ar@{->>}[u] &
=~D\sp{k} \ar[uu]\sp{\delta\sp{k}\sb{D}} \\
&&& \Omega\sp{n-k-1}\oW{n} \ar@{^{(}->}[u] \ar@/_{3.9pc}/[uu]\sb{0} & \\
C\sp{k-1}\Wu\sb{[n-1]}
\ar@{-->}[rrru]\sp{a\sp{k-1}} \ar[uu]\sp{\delta\sp{k-1}}
\ar@{.>}[rrr]_(0.5){F\sp{k-1}} &&& P\Omega\sp{n-k-1}\oW{n} \ar@{->>}[u] &
= D\sp{k-1} \ar[uu]\sp{\delta\sp{k-1}\sb{D}}
}
\noindent Note that \w{F\sp{k}} induces the map \w{a\sp{k-1}} in \wref[,]{complexes}
which must be nullhomotopic in order for \w{F\sp{k-1}} to exist.

Using the splitting from \wref{eq2} and Proposition \ref{lem7.2}, as
in \cite[Section 3]{BBSenH} we can show that there is a fibre sequence
\begin{myeq}\label{afibresequence}
\Sigma\Dup{n}\to\Wu\bp{n}\to \Wu\bp{n-1}\xrightarrow{F\sb{[n-1]}}\Dup{n}~,
\end{myeq}
\noindent where \w{\Dup{n}} is obtained by applying \w{\E\U} of \S \ref{scoch}
to the complex \w{\Dus}
in the right-hand side of diagram \wref{complexes} above, and the map
\w{F\bp{n-1}} is adjoint to the map of complexes given by \wref[.]{complexes}
\end{mysubsection}

\begin{prop}\label{pdiagdiff}
In the situation of \S \ref{scoqc}, an element \w{[f]\in E\sp{n,n+p}\sb{1}}
represented by \wref{eqeoneterm} survives to \w{E\sp{n,n+p}\sb{r+1}}
if and only if \wref{eqeoneterm} extends to a diagram:
\mytdiag[\label{eqerterm}]{
\Deln{n+r-1} \ltimes S\sp{p} \ar[r]\sb<<<<<<<<{g\sb{n+r-1}} &\bW\sp{n+r-1}\bp{n+r-1} \\
\Deln{n+r-2} \ltimes S\sp{p} \ar@<1.5ex>[u]\sp{d\sp{0}}\sb{\cdots}
\ar@<-1.5ex>[u]\sb{d\sp{n+r-1}} \ar[r]_>>>>>>>>{g\sb{n+r-2}}
\ar@{}[d]\sb{\vdots} &\bW\sp{n+r-2}\bp{n+r-2}
\ar@<1.5ex>[u]\sp{d\sp{0}}\sb{\cdots} \ar@<-1.5ex>[u]\sb{d\sp{n+r-1}}
\ar@{}[d]\sb{\vdots} \\
\Deln{n} \ltimes S\sp{p} \ar[r]\sb{g\sb{n}=f\sb{n}} & W\sp{n}\bp{n} \\
0\ar@<1.5ex>[u]\sp{d\sp{0}}\sb{\cdots} \ar@<-1.5ex>[u]\sb{d\sp{n}} \ar[r]
&\bW\sp{n-1}\bp{n} \ar@<1.5ex>[u]\sp{d\sp{0}}\sb{\cdots} \ar@<-1.5ex>[u]\sb{d\sp{n}}
}
\noindent indexed by \w[.]{\rDell{n-1,n+r-1}\times\bone}
\end{prop}

\begin{proof}
By induction on $r$, starting with \w[:]{r=1}

Because \w{\Wu} is Reedy fibrant, the natural map
\begin{myeq}\label{eqnormis}
\pi\sb{\ast}N\sp{n}(\Wu)~\to~ N\sp{n}(\pi\sb{\ast}\Wu)~\cong~C\sp{n}\pi\sb{\ast}(\Wu)
\end{myeq}
\noindent is an isomorphism, by \cite[X, 6.3(ii)]{BKanH} and Corollary \ref{cmoorenorm}.
By \wref[,]{eqceone} we can thus represent \w{[f]\in E\sp{n,n+p}\sb{1}} by a map
\w{\Du\ltimes S\sp{p} \to\Tot\sp{n}\Wu\bp{n}} \wwh in other words, by a sequence
of compatible maps \w{f\sb{i}:\Del\sp{i}\ltimes S\sp{p}\to\bW\bp{n}\sp{m}}
with \w{f\sb{i} = 0} for \w{m < n} (see \cite[\S 4.1]{BBSenH}). This determines a diagram
of pointed simplicial sets indexed by \w{\Dell{n}\times\bone} (in the notation of
\S \ref{snac}), whose restriction to \w{\rDele{n}} is depicted by
\mydiagram[\label{eqeoneterm}]{
\Deln{n} \ltimes S\sp{p} \ar[r]\sb(0.55){f\sb{n}} & \bW\sp{n}\bp{n}\\
\Deln{n-1} \ltimes S\sp{p} \ar@<1.5ex>[u]\sp{d\sp{0}}\sb{\cdots}
\ar@<-1.5ex>[u]\sb{d\sp{n}} \ar[r]\sb(0.55){0} \ar@{}[d]\sb{\vdots} & \bW\sp{n-1}\bp{n}
\ar@<1.5ex>[u]\sp{d\sp{0}}\sb{\cdots} \ar@<-1.5ex>[u]\sb{d\sp{n}} \ar@{}[d]\sb{\vdots} \\
\Deln{0} \ltimes S\sp{p} \ar[r]\sb(0.5){0} & \bW\sp{0}\bp{n}~,
}
\noindent which is equivalent to the case \w{r=1} of \wref[,]{eqerterm} since mapping
out of a zero object does not require choices of homotopies.

On the other hand, in such a commutative diagram, \w{f\sb{n}} factors through
\w{M\sp{n}\bp{n}\oW{n}} as in \wref[,]{eq2} and thus the diagram determines a map
\w{\Du\to\Tot\sp{n}(\Mu\bp{n}\oW{n})} by freely adding codegeneracies.
This in turn determines an element of \w[,]{\Omega\sp{n+p}N\sp{n}\Wu} since the fibre of
\w{\Wu\bp{n}\to\Wu\bp{n-1}} can be identified with \w{\Mu\bp{n}\oW{n}} by
\wref[.]{eq2}

Now assume the statement is true for \w[:]{r-1}
a map \w{g:\Du\to\Tot\sp{n+r-1}(\Wu\bp{n+r-1}} representing
an element of the \ww{E\sb{r-1}}-term of the spectral sequence survives to
the \ww{E\sb{r}}-term if and only if the composite \w{F\sb{[n-1]} \circ g} in
\mydiagram[\label{eqnullhtpyt}]{
  W\sb{[n+r]}\sp{\bullet} \ar[r] & W\sb{[n+r-1]}\sp{\bullet} \ar[rr]\sp{F\sb{[n-1]}} &&
  D\sb{[n-1]}\sp{\bullet} \\
\Du \ltimes S\sp{p}  \ar[ur]\sb{g} & &&
}
\noindent is nullhomotopic, where the horizontal maps are the fibre sequence
\wref[.]{afibresequence}

By the discussion in \S \ref{sdiagdes}, we see that both \w{F\sb{[n-1]}} and
\w{F\sb{[n-1]}\circ g} are adjoint, respectively, to right hand map and the
composite in the diagram:
$$
\Cus\F(\Du \ltimes S\sp{p}) \to\Cus\F \Wu\bp{n-1}\to\Dus
$$
They are thus adjoint to the maps $\psi$ and \w{\psi \circ \psi'} in the diagram below,
where the horizontal maps can be seen to be a sequence of fibrations.
$$
\xymatrix
{
\F W\sb{[n+r]} \ar[r] & \F W\sb{[n+r-1]} \ar[r]\sb{\psi} &\E \Dus \\
\Du \ltimes S\sp{p}  \ar[ur]\sb{\psi'} & &
}
$$
By adjunction, \w{F\sb{[n-1]}\circ g} is thus nullhomotopic if and only if the
composite \w{\psi \circ \psi'} is nullhomotopic, which holds if and only if
a diagram of the form \ref{eqerterm} exists.  Hence the result.
\end{proof}

\begin{corollary}\label{cor7.9}
Given $X$, $\fy$ and \w[,]{\xu} with \w{\Wu} a Reedy fibrant strictification of
\w[,]{\wu:=\MAP\sb{X}(\hy,\xu)\in c\Sa} as in \S \ref{scoqc}, a class
\begin{myeq}\label{eqfrepres}
[f]\in[\Sigma\sp{n+p}\fy,x\sp{n}]~=~\pi\sb{n+p}(w\sp{n})~\cong~
\pi\sb{p}(\Omega\sp{n}\bW\sp{n})~
\end{myeq}
\noindent in \w{E\sb{1}\sp{n,n+p}} of the associated spectral sequence survives
to \w{E\sb{r}\sp{n,p}} if and only if it fits into a map \w{B(\rDele{n+r-1}) \to\Sa}
of the form
\mytdiag[\label{eqerrterm}]{
\Deln{n+r-1} \ltimes\Sigma\sp{p}\fy\ar[r]\sb<<<<<<<<{g\sb{n+r-1}} & x\sp{n+r-1} \\
\Deln{n+r-2} \ltimes\Sigma\sp{p}\fy \ar@<1.5ex>[u]\sp{d\sp{0}}\sb{\cdots}
\ar@<-1.5ex>[u]\sb{d\sp{n+r-1}} \ar[r]\sb>>>>>>>>>{g\sb{n+r-2}} \ar@{}[d]\sb{\vdots} &
x\sp{n+r-2}
\ar@<1.5ex>[u]\sp{d\sp{0}}\sb{\cdots} \ar@<-1.5ex>[u]\sb{d\sp{n+r-1}}
\ar@{}[d]\sb{\vdots} \\
\Deln{n} \ltimes\Sigma\sp{p}\fy \ar[r]\sb{g\sb{n}=f\sb{n}} & x\sp{n} \\
\Deln{n-1} \ltimes\Sigma\sp{p}\fy\ar@<1.5ex>[u]\sp{d\sp{0}}\sb{\cdots}
\ar@<-1.5ex>[u]\sb{d\sp{n}}
\ar[r]\sb{0} 
& x\sp{n-2} \ar@<1.5ex>[u]\sp{d\sp{0}}\sb{\cdots}
\ar@<-1.5ex>[u]\sb{d\sp{n}} 
}
\end{corollary}

\begin{proof}
By \cite[Theorem 6.7]{Riehl1} and the fact that \w{\bW\bp{m}\sp{m} =\bW\sp{m}}
(by construction), a diagram \w{B(\rDele{n+r-1})\to\Sa} as in \wref{eqerrterm}
is equivalent to a homotopy coherent diagram of the form \wref[.]{eqerterm}
But such a diagram is equivalent to a strictly commuting diagram of pointed spaces
by a relative version of the usual  Dwyer-Kan homotopy coherence theorem
(see \cite{DKSmitH}). The result then follows from Proposition \ref{pdiagdiff}.
\end{proof}

\begin{mysubsection}{A combinatorial description of the differentials}
\label{scombdes}
By \cite{BKanS}, the unstable Adams spectral sequence for a space $\bX$ may be
identified with the homotopy spectral sequence of a certain cosimplicial resolution
\w{\Wu} of $\bX$, as in \cite[Chapter X]{BKanS}.

In \cite[Theorem 6.4]{BBSenH}, it was shown that applying the \ww{d\sb{r}}-differential
in this spectral sequence to an element of \w{E\sb{r}\sp{n,n+p}} represented by
\w{[f]\in E\sb{1}\sp{n,n+p}} as in Proposition \ref{pdiagdiff} yields a value
of a certain associated higher cohomology operation. This may be described more generally
in our situation as follows:

Assume we have lifted \w{[f]} as in \wref{eqeoneterm} to a map
\w{g\bp{N}:\Du\ltimes S\sp{p}\to\Wu\bp{N}} as in \wref[,]{eqerterm} for \w[.]{N=n+r-1}
From \wref{afibresequence} we see that \w{g\bp{N}} can be lifted to \w[,]{g\bp{N+1}}
up to homotopy, if and only if there is a nullhomotopy
\w[,]{H:F\bp{N}\circ g\bp{N}\sim 0} determined according to \S \ref{stott} by
a sequence of maps \w{H\sp{k}} fitting into a diagram
\myudiag[\label{eqnulhtpy}]{
\sms{C\Deln{k}}{S\sp{p}} \ar@/^{1.9pc}/[rrrrrd]\sp{H\sp{k}} &&&&&\\
& \hfsm{\Deln{k}}{S\sp{p}} \ar@{_{(}->}[ul]\sb{\delta\sp{0}} \ar[rr]\sp{g\bp{N}\sp{k}} &&
\bW\sp{k}\bp{N} \ar[rr]\sp(0.45){F\sp{k}} && P\Omega\sp{N-k-1}\oW{N+1} \\
\sms{C\Deln{k-1}}{S\sp{p}}
\ar@<2.5ex>[uu]\sp(0.55){Cd\sp{0}}\sp(0.45){=\delta\sp{1}}\sb(0.5){\dotsc}
 \ar@<-2.5ex>[uu]\sb(0.55){Cd\sp{k}}\sb(0.45){=\delta\sp{k+1}}
\ar@/^{3.3pc}/[rrrrrd]\sp{H\sp{k-1}} &&&&&\\
& \hfsm{\Deln{k-1}}{S\sp{p}} \ar@{_{(}->}[ul]\sb{\delta\sp{0}}
\ar@<2.5ex>[uu]\sp(0.5){d\sp{0}}\sb(0.5){\dotsc} \ar@<-2.5ex>[uu]\sb(0.5){d\sp{k}}
\ar[rr]\sp(0,55){g\bp{N}\sp{k-1}} &&
\bW\sp{k-1}\bp{N} \ar[rr]\sp(0.45){F\sp{k-1}} && P\Omega\sp{N-k}\oW{N+1}
 \ar@<3ex>[uu]\sp(0.6){\iota\sb{N-k-1}\circ p}\sb(0.6){=d\sp{0}}
 \ar@<-4ex>[uu]\sb(0.4){(j>0)}\sb(0.6){=d\sp{j}}\sp(0.6){0}
}
\noindent as in \cite[5.5]{BBSenH}. The value of the differential \w{\fd\sb{r}([f])}
is the obstruction to the existence of $H$; it is given by a certain map
\w[,]{\Phi:\sms{\partial\PPc\sp{N+1}\sb{r}}{S\sp{p}}\to\oW{N+1}}
depending only on the given maps \w{F\bp{N}} and \w[.]{g\bp{N}}
Here \w{\PPc\sp{N+1}\sb{r}} is a certain simplicial complex which is PL-equivalent to
an \wwb{N+1}-ball, so \w{\partial\PPc\sp{N+1}\sb{r}} is an $N$-sphere
(see \cite[Definition 5.1]{BBSenH}).

The proof of this result in \cite[Theorem 6.4]{BBSenH} used a specific CW construction
for \w[,]{\Wu} but in fact it is valid for any pair \w{\lra{\xu,\fy}} in a
quasi-category $X$ as in \S \ref{scoqc}. Moreover, the cell structure of
\w{\PPc\sp{N+1}\sb{r}} allows us to show that, up to homotopy, $\Phi$ is determined
inductively by the universal property of the colimit of these cells, providing
a model-independent and homotopy invariant description of the differentials in the
situation of Corollary \ref{cor7.9}, just as we did in \cite[Corollary 6.11]{BMeadS} in
the simplicial case.
\end{mysubsection}

%
%
\sect{The spectral sequence of a cosimplicial object and localization}
\label{ccsloc}
In this section, we obtain analogues of the constructions and results of
Section \ref{cssl} for the spectral sequence of a cosimplicial object \w{\xu}
in a quasi-category $X$.

First, note that from Corollary \ref{cor7.9} we may deduce the following analogue
of Proposition \ref{rtrunc}:

\begin{prop}\label{rctrunc}
Given \w{\fy\in X} and \w{\xu\in cX} as in \S \ref{scoqc} and
\w{[f]\in E\sb{1}\sp{n,n+p}} surviving to \w[,]{E\sb{r}} the element \w{\fd\sp{r}([f])}
in \w{E\sb{1}\sp{n+r,n+p+r-1}} associated to diagram \wref{eqnulhtpy}
\wwh and thus \w{d\sp{r}([f])\in E\sb{r}\sp{n-r,p+r-1}} itself
\wh depends only \w[,]{\Po{r-1}\Map\sb{X}(\Sigma\sp{p+n}\fy,\xd)\sb{f\sb{n}}} which
thus determines \w[,]{E\sb{r+1}} and in particular lets us decide
whether \w{[f]} survives to \w[.]{E\sb{r+1}}
\end{prop}

In order to obtain a cosimplicial version of Theorem \ref{tstem}, we need to replace
Definition \ref{dzmnp} by the following:

\begin{defn}\label{dgmnp}
Given \w{\fy\in X} as in \S \ref{sssso}, for any \w{n\geq 0} and \w{m,p\geq 1}
we let \w{G(n,m,\Sigma\sp{p}\hy)} denote the diagram
\mywdiag[\label{eqgmnp}]{
0 \ar@<2.5ex>[rr]\sp(0.3){d\sp{0}}\sb(0.3){\vdots} \ar@<-2.5ex>[rr]\sb(0.3){d\sp{n}} &&
\Sigma\sp{p}\hy\otimes \Deln{n}
\ar@<2.5ex>[rr]\sp(0.5){d\sp{0}}\sb(0.5){\vdots} \ar@<-2.5ex>[rr]\sb(0.5){d\sp{n+1}} &&
\Sigma\sp{p}\hy\otimes \Deln{n+1}\dotsc&
\Sigma\sp{p}\hy\otimes \Deln{n+m-1}
\ar@<2.5ex>[rr]\sp(0.5){d\sp{0}}\sb(0.5){\vdots} \ar@<-2.5ex>[rr]\sb(0.5){d\sp{n+m}} &&
\Sigma\sp{p}\hy\otimes \Deln{n+m}
}
\noindent in \w{\cpnX{n-1,n+m}} (unique up to weak equivalence), where we omit the
leftmost $0$ when \w[.]{n=0}
The notation \w{-\otimes\Deln{k}} serves merely as a placeholder,
to keep track of the cosimplicial identities. Set
\begin{myeq}\label{eqgr}
\G\sb{r}(\fy):=
\bigcup\sb{p\geq 1}~\bigcup\sb{n\geq 0}~\{G(n,r-1,\Sigma\sp{p}\hy)\}\cup
\bigcup\sb{1\leq m\leq\min\{p,r\}}
  \{G(n-m,m-1,\Sigma\sp{p-m+1}\hy)\}
\end{myeq}
\noindent with \w{\G(\fy)=\bigcup\sb{r=2}\sp{\infty}\,\G\sb{r}(\fy)} the collection
of all such diagrams, as in \S \ref{dzmnp}. The inclusions
\begin{myeq}\label{eqrestdiags}
\rDell{n,n+m}~\hookrightarrow~\rDell{n,n+m'}
\end{myeq}
\noindent again induce a partial order on the subset of diagrams in \w{\G(\fy)}
with a fixed $p$.  Note also that
\w{\Sigma G(n,m,\Sigma\sp{p}\fy)\simeq G(n,m,\Sigma\sp{p+1}\fy)} in
\w[.]{\cpnX{n-1,n+m}}

Thus given \w{\fy=\Sigma\hy} and \w{\xu} as in \S \ref{scoqc},
for each \w{r\geq 0} we define the cosimplicial \emph{$r$-stem} for
\w{\lra{\xu,\fy}} to be the system consisting of
\begin{myeq}\label{eqcstem}
\Po{m}\Map\sb{\cpnX{n-1,n+m}}(G(n,m,\Sigma\sp{p}\fy),\tau\sp{\ast}\xu)
\end{myeq}
\noindent for all \w[,]{G(n,m,\Sigma\sp{p}\fy)\in\G\sb{r}(\fy)}
under the various maps induced by \wref{eqrestdiags} and \wref[.]{eqstems}
Again, this is a more precise version of the ``spiral $r$-system'' of \cite[\S 5]{BBlanS}.
\end{defn}

\begin{mysubsection}{Cosimplicial stems and differentials}
\label{stail}
If $\mX$ is the simplicially enriched category corresponding to
the quasi-category $X$, as in \S \ref{spsss}, then not only the elements \w{[f]} in
\w[,]{E\sb{1}\sp{n,n+p}} but also \w{d\sb{1}([f])} (and thus the class of \w{[f]} in
\w[,]{E\sb{2}\sp{n,n+p}} if it survives) are determined by \w[.]{\Po{0}\mX\cong\ho X}

Once we know that \w{[f]} survives to \w{E\sb{r}\sp{n,n+p}} for \w[,]{r\geq 2} we know
from \S \ref{scombdes} that \w{\fd\sb{r}([f])} is determined by
the map \w{F:G(n,n+r-1,\Sigma\sp{p}\fy)\to\tau\sp{\ast}\xu} described by the part of
\wref{eqerrterm} in \w[.]{\cpnX{n,n+r-1}} Moreover, in this case the standard
simplicial enrichment of \w{\Sa} (see \cite[\S 1.5]{GJarS}) implies that all the
ingredients needed to describe the value of \w{\fd\sb{r}([f])} are contained in
\w{\Map\sb{\Sa}(S\sp{p},\bW\bp{n+k}\sp{n+k})\sp{i}} for \w{n\leq i\leq n+r-1}
(and \w[).]{0\leq k\leq r-1} In homotopy-invariant terms, this is determined by
\begin{myeq}\label{eqposstem}
\Po{r-1}\Map\sb{X}(\Sigma\sp{n+p}\hy,x\sp{n+k})\simeq
\Po{r+n+p-1}\Map\sb{X}(\hy,x\sp{n+k})\lra{n+p}\hsm \text{for}\ -1\leq k\leq r-1.
\end{myeq}
\noindent Fitting the spaces \wref{eqposstem} together to compute
\w{\fd\sb{r}([f])} is just taking the homotopy limit over the various coface maps to
calculate the mapping spaces of \wref[.]{eqcstem}

Thus we obtain the following analogue of Theorem \ref{tstem}:
\end{mysubsection}

\begin{thm}\label{tcstem}
Given \w{\fy\in X} and \w{\xu\in cX} as in \S \ref{scoqc} and \w[,]{r\geq 2}
the \ww{E\sp{r}}-term of the associated spectral sequence is determined
by the \wwb{r-2}stem of \w[.]{\lra{\xu,\fy}}
\end{thm}

This is a more precise version of \cite[Theorem 5.12]{BBlanS}.

\begin{mysubsection}{The Postnikov localization for cosimplicial objects}
\label{spcloc}
We can again reformulate Theorem \ref{tcstem} using the same pair of Bousfield
localizations on the quasi-category $Y$ of pointed $\infty$-groupoids
as in \S \ref{sploc}, except that in this case we use cosimplicial (rather than
cochain) windows, as in Remark \ref{rtruncations}:

If we set
\begin{myeq}\label{eqcproduct}
Z\sb{r}~:=~\prod\sb{G(n,m,\Sigma\sp{p}\hy)\in\G\sb{r}(\fy)}~\cpn{n-1,n-m+1}{Y}\bp{p}~,
\end{myeq}
\noindent as a quasi-category with fibrations (as in \S \ref{sploc}), the composites of
$$
\cpX~\xra{\Map\sb{X}(\hy,-)}~c\sp{+}Y~\xra{\tnk{n}{n-m+1}}~\cpn{n-1,n-m+1}{Y}
$$
\noindent combine to define a functor \w[,]{G\sb{r}:Z\sb{r}\to Z\sb{r}}
with the component of \w{G\sb{r}(\xu)} in \w{\cpn{n-1,n-m+1}{Y}\bp{p}} again
providing the \wwb{m,p,n}\emph{cosimplicial window} for \w[.]{\xu}

Again applying to \w{Z\sb{r}} the combined Bousfield localization
\w{\LLc\sp{m}\circ\R\sb{p}} of \S \ref{sploc} to each factor defines the $r$-th
\emph{Postnikov localization} functor \w[,]{\PPc\sp{r}:Z\sb{r}\to Z\sb{r}} where
\w{Z\sb{r}} has the structure of a quasi-category with fibrations and
\ww{\PPc\sp{r}}-equivalences.

Proposition \ref{rctrunc} again shows that the \ww{E\sp{r+2}}-page
of the spectral sequence for \w{\xu\in\cpX} is determined through
by \w[,]{G\sb{r}\xu} and we conclude from Theorem \ref{tcstem} the
following analogue of Corollary \ref{cstempr}:

\begin{corollary}\label{cctempr}
The \ww{\PPc\sp{r}}-equivalences induce isomorphisms of
the associated spectral sequences from the \ww{E\sb{r+2}}-term on.
\end{corollary}
\end{mysubsection}

\begin{mysubsection}{The $\E\sb{r}$-localization}
\label{scerl}
As in \S \ref{serl}, Assumption \ref{ass6.1}(1) for the Reedy structure on $X$
follows from the same assumption for \w{(X,\Cof,\W)} in \S \ref{ass6.1},
since cofibrations and weak equivalences in diagram categories are defined levelwise.
For (2), a family of generating (trivial) cofibrations for the Reedy
structure can be produced from a family of generating (trivial) cofibrations for
\w{(X,\Cof,\W)} by a standard argument.
For (3), we can identify the localization map \w{\cpX\to \LR(\cpX)} with the map
\w{\cpX\to c\sp{+}\LWX} by \cite[Theorem 7.6.17]{CisiH}.
Thus, the localization map is accessible, since colimits of presheaf categories
can be calculated pointwise by \cite[5.1.2.2]{LurieHTT}.

Therefore, if we let \w{\E\sb{r}} denote the set of left Kan extensions of \w{\G\sb{r}}
(see \S \ref{dgmnp}) from the relevant truncations of \w[,]{\cpX} as in \S \ref{serl},
Theorem \ref{thm6.7} implies that \w{\cpX} has the structure of a quasi-category with
cofibrations and weak equivalences, in which the latter are the
\ww{\E\sb{r}}-equivalences \wh that is,
\ww{\E\sb{r}}-local maps \w{\xu\to\yu} in \w[.]{\cpX}  We denote by \w{R\sb{\E\sb{r}}} the right Bousfield
localization of \w{\cpX} with respect to \w[.]{\E\sb{r}}
\end{mysubsection}

We deduce from Corollary \ref{cor7.9}:

\begin{corollary}\label{thm8.6}
A \ww{\E\sb{r}}-equivalence \w{f: \xu \to \yu} in \w{\cpX} induces a bijection of
the associated  spectral sequences at the \ww{E\sb{r}}-term.
\end{corollary}

\end{document}